%% file: BvM.tex
\documentclass[fleqn]{scrartcl}

\usepackage{graphicx}
\usepackage{amsmath}
\usepackage{amssymb}
\usepackage{tikz}
\usepackage{amsthm}
\usepackage{natbib}
\usepackage{enumerate}
\usepackage{subfigure}
\usepackage{comment}
\usepackage{enumitem}
\usepackage{url}
\definecolor{otto}{RGB}{210,49,91} 
\definecolor{kiwi}{RGB}{3,80,133}  
\usepackage[colorlinks,linkcolor=kiwi,citecolor=otto,urlcolor=kiwi]{hyperref}
\usepackage{dsfont} 
\usepackage[lined,boxed,linesnumbered,ruled]{algorithm2e}
\usepackage{calc}
\usepackage{bm} 

\renewcommand{\ge}{\geqslant}
\renewcommand{\geq}{\geqslant}
\renewcommand{\le}{\leqslant}
\renewcommand{\leq}{\leqslant}
\newtheorem{Theorem}{Theorem}[]
\newtheorem{Lemma}[Theorem]{Lemma}
\newtheorem{Remark}[Theorem]{Remark}
\newtheorem*{Remark*}{Remark}
\newtheorem{Corollary}[Theorem]{Corollary}

\newtheorem{Example}{Example}
\newtheorem*{Abstract}{Abstract}

\newcounter{ex_mean}
\newcounter{ex_linear}
\newcounter{temp}
\newtheorem{Assumption}{Assumption}

\title{Asymptotic considerations in a Bayesian linear model with nonparametrically modelled time series innovations }
\author{  {Claudia Kirch$^{1,2}$}  {Alexander Meier$^{1}$} {Renate Meyer$^{3}$} {Yifu Tang$^{4}$}}

\begin{document}
\input{cmd}
\maketitle

\begin{Abstract}
This paper  considers a semiparametric approach within the general Bayesian linear model where the innovations  consist of a stationary, mean zero Gaussian time series. While a parametric prior is specified for the linear model coefficients,
the autocovariance structure of the time series is modeled nonparametrically using
 a Bernstein-Gamma process prior for the spectral density function, the Fourier transform of the autocovariance function.
 When updating this joint prior 
 with Whittle's likelihood, 
a Bernstein-von-Mises result is established for the linear model coefficients  showing the asymptotic equivalence of the corresponding estimators to those obtained from frequentist pseudo-maximum-likelihood estimation under the Whittle likelihood.
Local asymptotic normality of the likelihood is shown, demonstrating that the marginal posterior distribution of the linear model coefficients 
shrinks at parametric rate towards the true value, and that the conditional posterior distribution of the spectral 
density contracts in the sup-norm, even in the case of a partially misspecified
linear model.

\end{Abstract}
\footnotetext[1]{Department of Mathematics, Otto-von-Guericke University,  Magdeburg, Germany. }
\footnotetext[2]{Corresponding author: \url{claudia.kirch@ovgu.de}.} 
\footnotetext[3]{Department of Statistics, University of Auckland, 38 Princes Street, Auckland 1010, New Zealand. \url{renate.meyer@auckland.ac.nz}.}
\footnotetext[4]{Department of Mathematics and Statistics, University of Otago, New Zealand. \url{yifu.tang@otago.ac.nz}.}


\textbf{MSc2010 subject classification:}  62M10 (Time series, auto-correlation, regression, etc. in statistics);
62G86 (Nonparametric inference and fuzziness);
62F15 (Bayesian inference)

\section{Introduction}
Bayesian nonparametric approaches to the analysis of stationary time series have received a lot of attention over the last two decades, both from an applied and theoretical perspective. The majority of these approaches aim at estimating the spectral density and are thus formulated in the frequency domain, utilizing the Whittle likelihood approximation and placing a nonparametric prior on the (log-) spectral density function, e.g.\ \cite{CarterKohn97,choudhuri,rousseau2012, Chopin13, cadonna2017,kirch2019, edwards2019, rao2021, russel2021}. For Bayesian parametric methods, the Bernstein-von-Mises (BvM) theorem  guarantees that frequentist maximum-likelihood and Bayesian inference procedures are asymptotically equal to the second order of approximation. But that is no longer true for Bayesian nonparametric models. One even needs to explicitly show posterior consistency, i.e.\ that the posterior distribution contracts around the true parameter value with increasing sample size.
Proofs of the consistency of the respective posterior distribution  of the spectral density have been provided by  \cite{choudhuri}, \cite{rousseau2012}, \cite{kirch2019} for zero-mean Gaussian time series and by \cite{TangYifu2023Pcft} for zero-mean non-Gaussian time series.

Arguably, (linear) regression models are among the most widely used statistical procedures where the standard assumption of independent innovations does not always hold in practical applications. More realistically, a mean-zero Gaussian time series can be  used to model the correlated innovations in such a regression model, as  for instance in \cite{CarterKohn97,DeyTanujit2018Btsr,JunYoonBae2022Beot,merkatas2022identification}. In such a semiparametric model, the aim is to estimate the regression coefficients jointly with the spectral density of the noise process, the Fourier transform of the autocovariance function. In this situation, the spectral density is of secondary interest and could be regarded as a nuisance parameter.  However, a proper and adapt modeling of the autocorrelation structure has the potential  to properly capture not only the first but also second-order properties.  A parametric model for the autocorrelation such as the classical assumption of i.i.d.\ Gaussian noise, even though first-order consistent, would over- or underestimate the variability if the parametric model is misspecified. This is problematic not only in view of understanding the uncertainty associated with the obtained estimator but also for understanding the uncertainty associated with corresponding predicted values. This can be overcome by using a nonparametric modeling of the error structure.

In this paper we consider the semiparametric general linear model with time series errors and estimate both linear regression coefficients and autocorrelation structure in the frequency domain, using a Bernstein-Gamma process prior for the spectral density. The main focus is on establishing a Bernstein-von-Mises result for the marginal posterior distribution of the regression coefficients. In particular, we show that we achieve the same contraction rate as with a parametric model for the autocovariance and derive the corresponding asymptotic variance of the procedure based on the Whittle likelihood approximation. For the special case of a general linear model with design matrix consisting of a vector of ones and time series errors, i.e.\ estimating the mean of a stationary time series, we can show that the asymptotic variance is the same as that of the maximum likelihood estimator based on the true likelihood. However, we also provide a counterexample where the variance of the estimator under the true likelihood differs from that under the Whittle likelihood so that here, too, the variability of the estimator is not properly captured by the Bayesian procedure. This shows, that one needs to be careful with the interpretation of the Bayesian results and, if possible, check validity on a case-by-case basis.

The remainder of the paper is organized as follows: In Section~\ref{sec_2}, we give a definition of the semipametric model including likelihood and prior specifications and state the Bernstein-von Mises theorem. Section~\ref{sec:simulation} provides results of a comprehensive simulation study and a case study. The proof of the Bernstein-von Mises theorem is given in Section~\ref{sec:proof} by first stating a general semiparametric Bernstein-von Mises theorem, which can be applied to obtain the main result of this paper by establishing a LAN property and joint posterior contraction rates. Technical lemmas are given in Appendix~\ref{sec:technical}.

\section{Semiparametric model and Bernstein-von-Mises theorem}\label{sec_2}
In this section, we propose a Bayesian procedure for a linear model with time-dependent errors, where we do not use a parametric time series model. Instead, we use a Whittle-likelihood approximation of the time-dependent errors in combination with a Bernstein-Gamma process prior for the spectral density. We show consistency and parametric contraction rates for the linear regression coefficients and derive their asymptotic variance as well as a Bernstein-von-Mises theorem.

\emph{Notation:}
We denote the Euclidean norm of a vector~$\vect x$ by~$\|\vect x\|$.
Similarly, for a matrix~$\matr A$, we denote by~$\|\matr A\|$ the Euclidean
matrix norm (i.e.~the largest singular value which is given by $\sqrt{\lambda_{max}(AA^T)}$) and by~$\|\matr A\|_F$ the Frobenius norm. Let  $S_n^+$ denote the space of symmetric positive
definite (spd)~$n\times n$ matrices.
For a $\matr B\in S_n^+$, we denote by~$\matr B^{1/2}$ the (unique) spd matrix
square root. Furthermore, we denote $x_n \lleq y_n$, if there exists a constant $C>0$ such that $x_n\le C y_n$. Similarly, $x_n\ggeq y_n$, if there exists a constant $c>0$ such that $x_n\ge c y_n$. We denote the minimal respectively maximal eigenvalue of a matrix $\matr A$ by $\lambda_{\min}(\matr A)$ respectively $\lambda_{\max}(\matr A)$. A lower index $j$ behind a vector refers to its $j$th component.

\subsection{Linear model with nonparametric time series errors}\label{subsec:definition}
Consider the following linear model 
\begin{align*}
	Z_t = \vect x_t^T\vect\theta + e_t, \quad\text{for } t=1,\ldots,n, \end{align*}

Unlike in a classical setting, we do not assume independence of the innovations~$\{e_t\}$, instead, we allow for time-dependency by modelling them as a stationary mean zero Gaussian time series
with covariance matrix~$\matr\Sigma_n\in\mathcal S_n^+$.  Equivalently, this model can be stated in  vectorized notation
\begin{align}\label{eq:linearModelTS}
  \vect Z_n = \matr X_n \vect\theta + \vect e_n, \quad \vect e_n \sim N_n(\vect 0, \matr\Sigma_n),
\end{align}
with design matrix~$\matr X_n \in \mathbb R^{n\times r}$.

Corresponding methodology can also be developed for a multivariate regression model by replacing scalar calculations by corresponding matrix-vector calculations similarly to \cite{meierPaper}. For simplicity of presentation we concentrate on the univariate situation here.

We make the following assumptions on the regression part of the model:

\begin{Assumption}\label{ass:designMatrix}
The design matrix~$\matr X_n\in\mathbb R^{n\times r}$ is of full rank for~$n$ sufficiently large.
Furthermore, we assume that $0<\liminf_{n\to\infty}\frac{1}{n}\lambda_{\min}(\matr X_n^T\matr X_n)\leqslant \limsup_{n\to\infty}\frac{1}{n}\lambda_{\max}(\matr X_n^T\matr X_n)<\infty$. 
\end{Assumption}
Assumption~\ref{ass:designMatrix} implies that $\lambda_{\min}(\matr X_n^T \matr X_n)\to\infty$ and $\|\matr X_n\|=o(\exp( n^{\delta}))$ for every~$\delta>0$ as~$n\to\infty$. These implications are useful in the proof of the main result. 

Often, in regression analysis, the  design matrix 
fulfils
\begin{align}\label{stand_ass_regr}
	\frac{1}{n}\matr X_n^T\matr X_n\to\matr V \quad \text{for some matrix }\matr V\in\mathcal S_r^+.
\end{align}
In this case, Assumption~\ref{ass:designMatrix} holds since by the continuity of eigenvalues,  $\matr X_n^T\matr X_n$ is positive definite and has rank $r$ for  sufficiently large $n$,
\begin{align*}
	\frac{1}{n}\lambda_{\max}(\matr X_n^T\matr X_n)\lleq \lambda_{\max}(\matr V)<\infty
\end{align*}
and similarly
\begin{align*}
	\frac{1}{n}\lambda_{\min}(\matr X_n^T\matr X_n)\ggeq \lambda_{\min}(\matr V)>0.
\end{align*}

Additionally, we make the following assumption on the time series:
\begin{Assumption}\label{ass:fModel}
The true spectral density~$f_0$ is uniformly bounded and uniformly bounded away from~0.
The corresponding autocovariance function~$\gamma_0(h)=\int_0^{2\pi}f_0(\omega)\exp(ih\omega)d\omega$ fulfils
\[
  \sum_{h\geq 0}|\gamma_0(h)|\,h^a<\infty
\]
for some~$1<a\leqslant2$.
\end{Assumption}
This is an assumption on the decay of the autocovariance function of the time series and implies that the spectral density is differentiable with $(a-1)$-H\"older-continuous derivative (see e.g. Remark 8.5 in \cite{Serov}).
Under this assumption the Whittle likelihood  below  and the true likelihood are mutually contiguous. Contiguity is useful to prove posterior consistency of the spectral density. Using different proof techniques as e.g.\ in \cite{TangYifu2023Pcft} the assumption on the time series can be relaxed.

We now give some important examples of such linear models in the context of time series analysis:
\begin{Example}[Time series with unknown mean]\label{ex_mean_model}
The simplest situation covered by this methodology deals with Bayesian inference about the mean of a time series with no parametric assumptions on the structure of the time series. In this case, $Z_t=\mu+e_t$, i.e.\ $\matr X_n=(1,\ldots,1)^T\in\mathbb R^{n\times 1}$
and~$\theta=\mu$. Furthermore, we observe that~$\matr X_n^T\matr X_n=n$ such that \eqref{stand_ass_regr} and thus
Assumption~\ref{ass:designMatrix} are clearly fulfilled.
\end{Example}
\setcounter{ex_mean}{\value{Example}-1}
\begin{Example}[Time series with unknown linear trend]\label{ex_linear}
In the linear trend model~$Z_t=\mu+\frac{bt}{n}+e_t$, $t=1,\ldots,n$, we are interested in Bayesian inference for a linear trend in a time series.
In this case, it holds~$\vect\theta=(\mu,b)^T$ and
\[
  \matr X_n=\begin{pmatrix}
    1 & 1 & \ldots & 1 \\
    \frac{1}{n} & \frac{2}{n} & \ldots & 1
  \end{pmatrix}^T, \quad
  \matr X_n^T\matr X_n=\begin{pmatrix}
    n & \frac{n+1}{2} \\
    \frac{n+1}{2} & \frac{(n+1)(2n+1)}{6n}
  \end{pmatrix}.
\]
Clearly,
\[
  \frac{1}{n} \matr X_n^T\matr X_n \to \begin{pmatrix} 1 & \frac{1}{2} \\ \frac{1}{2} & \frac{1}{3} \end{pmatrix} \in \mathcal S_2^+,
\quad n \to \infty,
\]
such that \eqref{stand_ass_regr} and thus
 Assumption~\ref{ass:designMatrix} are fulfilled.
\end{Example}

\setcounter{ex_linear}{\value{Example}-1}
\begin{Example}[Time series with unknown linear trend and seasonality]\label{example:sesonality}
A time series model with linear trend and seasonality can be written as $Z_t=\mu + b\frac{t}{n}+\sum_{l=1}^{p-1} s_l1_{\{t\in M_l\}} + e_t$, where $p$ is the period, $M_l=\{l+i\cdot p: i=0,1,2,\ldots\}$ and $\vect\theta=(\mu,b,s_1,\cdots,s_{p-1})^T$. The last seasonal component has been omitted due to identifiability with the intercept and for $p=1$ the model coincides with Example~\ref{ex_linear}. Thus, the design matrix is given by $\matr X_n=(\mathbf{y}_0,\mathbf{y}_1,\ldots,\mathbf{y}_p)$, where $\mathbf{y}_0^T=(1,1,\ldots,1)$, $\mathbf{y}_1^T=(1,2,\ldots,n)/n$ and $\mathbf{y}_l$, $l=1,\ldots,p-1$, are vectors with entries of $1$ at the positions $l+i\cdot p$, $i\ge 0$, and $0$ at all other positions. Some calculations yield
\begin{align*}
    & \frac{1}{n} \matr X_n^T\matr X_n \to \begin{pmatrix} 
    1& \frac{1}{2} &\frac{1}{p} &\ldots &\ldots&\frac{1}{p} \\[1mm] 
    \frac{1}{2} & \frac{1}{3}& \frac{1}{2p}&\ldots &\ldots&\frac{1}{2p} \\[1mm]
    \frac{1}{p}&  \frac{1}{2p} &\frac{1}{p}& 0& \ldots&0 \\
     \vdots & \vdots&0&\ddots&\ddots&0 \\
\vdots & \vdots&\vdots&\ddots&\ddots&0 \\
\frac{1}{p} & \frac{1}{2p}& 0&\ldots&\ldots&\frac 1 p
    \end{pmatrix} \in \mathcal S_{p+1}^+,
\quad n \to \infty,
\end{align*}
which can be seen to be positive definite by calculating the determinant using a result for partitioned matrices given e.g.\ in 
\cite{Bernstein20181}, Prop.\ 3.9.4. Consequently, \eqref{stand_ass_regr} and thus
 Assumption~\ref{ass:designMatrix} are fulfilled.

\end{Example}
\subsection{Pseudo-Likelihood and Prior specification}\label{section_prior}
Modelling the distribution of the Gaussian stationary time series errors nonparametrically requires specifying a nonparametric prior on the space of symmetric,  positive semidefinite covariance matrices. Furthermore, the evaluation of the Gaussian likelihood subsequently requires the inversion of a high-dimensional covariance matrix which is computationally expensive. In the literature, this is often avoided by making use of the so-called Whittle likelihood, which is based on the observation that 
a Fourier transform of the original time series observations  leads to (in a certain sense) asymptotically independent complex Gaussians with variances equal to the spectral density evaluated at the Fourier frequencies.  The corresponding approximate Whittle likelihood is misspecified  but asymptotically contiguous to the true likelihood under the Assumption 2 as shown by \cite{choudhuri}.

More explicitly,  Whittle's likelihood~$P_W^n$ is formulated
in terms of the Fourier coefficients~$\vectt e_n \in \mathbb R^n$ of the innovations
as obtained via the transformation~$\vectt e_n = \matr F_n \vect e_n$,
where~$\matr F_n\in\mathbb R^{n\times n}$ is the (orthogonal) discrete
Fourier transformation matrix, defined in \eqref{eq_formular_Fn} below,
see Section~2.1 in~\cite{kirch2019} for more details.
Whittle's likelihood for the observations~$\vect Z_n$ can then be written in terms 
of the following Lebesgue density (see (8.3) in~\cite{meier} for more details)
\begin{equation}\label{eq:whittleLinearModel}
  p_W^n(\vect z_n|\vect\theta,f) =
  \frac{1}{\sqrt{(2\pi)^n|\matr D_n|}}\exp\left( -\frac{1}{2}\left( \vectt z_n-\matrt X_n\vect\theta \right)^T \matr D_n^{-1} \left( \vectt z_n-\matrt X_n\vect\theta \right) \right)
\end{equation}
with diagonal frequency domain covariance matrix~$\matr D_n:=\matr D_n(f)$ defined by 
\begin{equation}\label{eq_def_Dn}
	\matr D_n=\matr D_n(f):=2\pi\,\begin{cases}
	\text{diag}(f(0),f(\lambda_1),f(\lambda_1),\ldots,f(\lambda_N),f(\lambda_N),f(\lambda_{n/2}))& n \text{ even},\\
	\text{diag}(f(0),f(\lambda_1),f(\lambda_1),\ldots,f(\lambda_N),f(\lambda_N))& n \text{ odd}.
\end{cases}
\end{equation}
with  Fourier frequencies $\lambda_j=\frac{2\pi j}{n}$, for $
j=0,\ldots,N=\lfloor(n-1)/2 \rfloor$,
 the frequency domain design matrix~$\matrt X_n:=\matr F_n\matr X_n$
and the frequency domain observations~$\vectt Z_n=\matr F_n\vect Z_n$, where
\begin{align}\label{eq_formular_Fn}
		F_n=\begin{cases}
		(\mathbf{e}_0,\mathbf{c}_1,\mathbf{s}_1,\ldots,\mathbf{c}_N,\mathbf{s}_N,\mathbf{e}_{n/2})^T, & n\text{ even},\\
		(\mathbf{e}_0,\mathbf{c}_1,\mathbf{s}_1,\ldots,\mathbf{c}_N,\mathbf{s}_N)^T, & n\text{ odd},
	\end{cases}
\end{align}
where $\mathbf{c}_j=\sqrt{2} \Re \mathbf{e}_j$ (real part), $\mathbf{s}_j=\sqrt{2}\Im \mathbf{e}_j$ (imaginary part) and 
$\mathbf{e}_j=n^{-1/2}(e_j,e_j^2,\ldots,e_j^n)^T$ with $e_j=\exp(-2\pi i j/n)$.

\par
The Bayesian model specification
is completed by specifying a prior for the linear model coefficient and innovation
covariance structure, which will be modeled in terms of the spectral density function~$f$.
We consider an independent prior of the form~$P(d\vect\theta,df)=P(d\vect\theta)P(df)$.
For the spectral density, we will employ a Bernstein-Gamma prior, a special case of the Bernstein-Hpd-matrix-Gamma prior, developed for multivariate time series  in
\cite{meierPaper} and used here for the univariate case. The main idea is  to represent the spectral density by a positive linear combination of Bernstein polynomials (based on Weierstrass approximation theorem) and putting a prior on the positive weights and the degree of Bernstein polynomials. Note that suitably normalized, the Bernstein polynomials on the interval $[0,1]$ are Beta densities.

For~$0<\tau_0<\tau_1<\infty$, consider the set
\begin{equation}\label{eq:truncationSetForBvm}
  \mathcal C_{\tau_0,\tau_1} := \left\{ f \in \mathcal C^1 \colon \|f^{-1}\|_\infty \leq \tau_0^{-1}, \|f\|_L\leq\tau_1 \right\}
\end{equation}
with the Lipschitz norm~$\|f\|_L:=\|f\|_\infty+\|\Delta f\|_\infty$,
where~$\Delta f(\omega_1,\omega_2):=|f(\omega_1)-f(\omega_2)|/|\omega_1-\omega_2|$
for~$\omega_1\neq\omega_2$.
Recall the univariate Gamma process~$\operatorname{GP}(\alpha,\beta)$ on~$\mathcal X=[0,\pi]$ with~$\alpha,\beta\colon [0,\pi]\to (0,\infty)$ (see e.g.~(1.9) in~\cite{meier}).
Let~$I_{j,k}:=( (j-1)\pi/k,j\pi/k ]$ for~$j=1,\ldots,k$ denote the equidistant partition of~$[0,\pi]$ of size~$k$
and denote by~$b(\cdot|j,l)$ the probability density function on~$[0,1]$
of the~$\operatorname{Beta}(j,l)$ distribution. Let $P(df)$ denote 
the Bernstein-Gamma prior for~$f$ defined  by:
\begin{align*}
  f(\omega)= f(\omega|\Phi, k) :=\sum_{j=1}^k\Phi(I_{j,k})b(\omega/\pi|j,k-j+1), \quad 0 \leq \omega \leq \pi, \\
  \Phi \sim \operatorname{GP}(\alpha,\beta), \quad k \sim p(k).
\end{align*}
The restriction of~$P(df)$ to~$\mathcal C_{\tau_0,\tau_1}$ will be denoted  by~$P_{\tau_0,\tau_1}(df)$:
\begin{equation}\label{eq:truncationForBvMSet}
  P_{\tau_0,\tau_1}(F) := \frac{P(F\cap\mathcal C_{\tau_0,\tau_1})}{P(\mathcal C_{\tau_0,\tau_1})}, \quad F \subset \mathcal C^1 \text{ measurable}.
\end{equation}
The joint prior for~$(\vect\theta,f)$ will be denoted 
by~$P_{\tau_0,\tau_1}(d\vect\theta,df):=P(d\vect\theta)P_{\tau_0,\tau_1}(df)$.
We will denote by~$P_{W;\tau_0,\tau_1}^n(d\vect\theta,df|\vect Z_n)$ the pseudo posterior
for~$(\vect\theta,f)$
obtained when updating~$P_{\tau_0,\tau_1}$ with Whittle's likelihood~$P_W^n$ from~\eqref{eq:whittleLinearModel}.

We assume that the prior for $\vect\theta$ is continuous, has moments up to a certain order, and its support includes the true parameter vector $\vect\theta_0$, i.e., more specifically: 
\begin{Assumption}\label{ass:thetaPrior}
	The prior~$P(d\vect\theta)$ on~$\vect\theta$ is such that $\E \|\vect\theta\|^{\nu} <\infty $ for some $\nu>0$. It possesses a Lebesgue density~$g$ on~$\mathbb R^r$
fulfilling~$g(\vect\theta)\leq C$ (for some $C>0$) and is continuous and strictly positive in an open
neighbourhood of~$\vect\theta_0$ .
\end{Assumption}

Furthermore, we need the following two assumptions about the hyperparameters of the nonparametric prior on the spectral density function $f$:

\begin{Assumption}\label{ass:GPBVM}
  For the Gamma process parameters~$\alpha,\beta\colon [0,\pi]\to (0,\infty)$,
  there exist positive constants~$g_0,g_1$ such that~$g_0 \leq \alpha(x),\beta(x) \leq g_1$
  holds for all~$x \in [0,\pi] \setminus N$, where~$N$ is a Lebesgue null set.
\end{Assumption}

\begin{Assumption}\label{ass:kPrior}
There exist positive constants~$A_1,A_2,\kappa_1,\kappa_2$ such that
the prior probability mass function $p(\cdot)$ of the Bernstein polynomial degree~$k$ fulfils
\[
  A_1\exp(-\kappa_1j\log j)\leq p(j)\leq A_2\exp(-\kappa_2j)
\]
for~$j\in\mathbb N$.\end{Assumption}

Assumption \ref{ass:GPBVM} is needed to ensure that any small neighbourhood of the true spectral density has positive prior probability.  Assumption \ref{ass:kPrior} ensures that the prior on $k$ has full support on $\mathbb N$ and that the decrease of tail probabilities is bounded.
\subsection{Bernstein-von-Mises result}
We are now ready to state a Bernstein-von-Mises result for the coefficient of the linear part of the model.

\begin{Theorem}\label{eq:bvmLinearModel}
	Let \eqref{eq:linearModelTS} hold with fixed design matrix~$\matr X_n$ fulfilling Assumption~\ref{ass:designMatrix}
and stationary mean zero Gaussian innovations $\{e_t\}$ with
spectral density~$f_0$ fulfilling Assumption~\ref{ass:fModel}.
Denote the true distribution of~$\vect Z_n$ by~$P_0^n$ and let $0<\tau_0<\min_{\omega}f_0(\omega)\le \max_{\omega}f_0(\omega)<\tau_1<\infty$. Furthermore, let the prior specification hold as in Section~\ref{section_prior} fulfilling Assumptions~\ref{ass:thetaPrior}--\ref{ass:kPrior}.
Then, it holds, as $n\to\infty$, in $P_0$-probability
\begin{align*}
	&\delta_{\operatorname{TV}}\left[ P_{W;\tau_0,\tau_1}^n(\vect\theta\in\cdot|\vect Z_n), N_r\left(\matr V_{W;n}\, \matrt X_n^T\matr D_{n,0}^{-1} \vectt Z_{n}, \matr V_{W;n}  \right) \right] \to 0,\\
 &\text{where }\matr V_{W;n}=\left( \matrt X_n^T\matr D_{n,0}^{-1}\matrt X_n \right)^{-1},\quad \matr D_{n,0}=\matr D_n(f_0)
\end{align*}
and
$\delta_{\operatorname{TV}}$ denotes the total variation
distance. \end{Theorem}

Because $\matr V_{W;n}$ typically contracts like $1/n$ by Assumptions \ref{ass:designMatrix} or by \eqref{stand_ass_regr}, the theorem shows that we have contraction at parametric rate (compare also Section~\ref{sec:marginalParametric}).

\subsubsection{Asymptotic second-order correctness}
\label{rem_whittle_ML}
	The quantity $\matr V_{W;n}\, \matrt X_n^T\matr D_{n,0}^{-1} \vectt Z_{n}$ in Theorem~\ref{eq:bvmLinearModel} is the pseudo maximum likelihood estimator under the Whittle likelihood with known spectral density, which is a special case of weighted least squares estimation. As a frequentist estimator this quantity is unbiased for $\vect\theta_0$ and has a normal distribution with covariance $\matr V_{W;n}$ under the Whittle likelihood. This shows that the posterior covariance matrix is asymptotically equivalent  to the one of the pseudo maximum likelihood estimator  under the Whittle likelihood.
 
In contrast, the covariance of the latter estimator under the true likelihood is given by 
\begin{align}\label{eq_cov_true}
	&\matr V_{0;n}=\left( \matrt X_n^T\matr D_{n,0}^{-1}\matrt X_n \right)^{-1}\,\matrt X_n^T\matr D_{n,0}^{-1}\matrt \Sigma_n
	\matr D_{n,0}^{-1}\matrt X_n\,\left( \matrt X_n^T\matr D_{n,0}^{-1}\matrt X_n \right)^{-1},
\end{align}
where $\matrt\Sigma_n=\matr F_n \matr\Sigma_n\matr F_n^T$. 

If $\matr V_{0;n}$ and $\matr V_{W;n}$ are asymptotically equivalent in the sense of 
\begin{align}\label{eq_second_order}
\|\matr V_{W;n}^{-1/2}\matr V_{0;n}\matr V_{W;n}^{-1/2}-\mbox{Id}\|_F=o(1),
\end{align}
then the total variation distance between two Gaussians with the same mean and $\matr V_{0;n}$ respectively $\matr V_{W;n}$ as covariance matrices converges to zero. Indeed, this is not only a sufficient but also a necessary condition because the minimum of $1$ and the left hand side of \eqref{eq_second_order}
 upper and lower bounds the total variation distance  (with appropriate factors) between such two multivariate Gaussians (see e.g.\ \cite{frob}).
 
Consequently,  under \eqref{eq_second_order},
we can replace the statement in Theorem~\ref{eq:bvmLinearModel} by a corresponding statement with $N_r\left(\matr V_{W;n}\, \matrt X_n^T\matr D_{n,0}^{-1} \vectt Z_{n}, \matr V_{0;n}  \right)$, i.e.\ the covariance matrix under the true likelihood instead of under the Whittle likelihood. 
This shows second order correctness of the Bayesian procedure in the sense that the credibility sets are asymptotic confidence sets in this case. Otherwise any Bayesian uncertainty quantification of a Bayesian point estimate such as the posterior standard deviation will systematically over- or underestimate the true uncertainty.

In the one-parameter situation, i.e. when $\matr X_n$ is actually a vector,  under the boundedness assumption of Assumption~\ref{ass:fModel} in combination with \eqref{stand_ass_regr} the condition in \eqref{eq_second_order}  comes down to 
\begin{align}\label{eq_asym_equiv}
  &  \left\|\frac 1 n\matrt X_n^T\matr D_{n,0}^{-1}(\matrt \Sigma_n-\matr D_{n,0})
	\matr D_{n,0}^{-1}\matrt X_n\right\|_F\notag\\
&	=\left\|\frac 1 n\matr X_n^T\left(\matr F_n^T\matr D_{n,0}^{-1}\matr F_n\, \matr \Sigma_n\, \matr F_n^T\matr D_{n,0}^{-1}\matr F_n- \matr F_n^T\matr D_{n,0}^{-1}\matr F_n \right) \matr X_n\right\|_F
	\to 0.
\end{align}
Standard results from time series analysis (or circulant matrices) show that  the entries of $\matrt \Sigma_n-\matr D_{n,0}$ are uniformly bounded by $1/n$  under suitable assumptions (see e.g.\ Theorem 4.9 in~\cite{meier}), which is not strong enough to obtain the above results in general. 

Indeed, Example~\ref{counterexample_1} gives a counter-example, where asymptotic second-order correctness does indeed not hold, while Examples~\ref{ex_mean_model} and \ref{ex_linear} give situations, where asymptotic second-order correctness does hold.

This shows that one has to check asymptotic second-order correctness on a case-by-case basis for the credibility sets to be asymptotic confidence sets.

We will now come back to our initial examples:
\setcounter{temp}{\value{Example}}
\setcounter{Example}{\value{ex_mean}}
\begin{Example}[Time series with unknown mean (cont.)]
 Some simple calculations show that $\matrt X_n=(\sqrt{n},0,\ldots,0)^T$ and the first component of $\vectt Z_{n}$ is given by the sample mean $\bar{Z}_n$ of the observations. Consequently, by Theorem~\ref{eq:bvmLinearModel} the posterior distribution of $\mu$ is centered at the sample mean and multiplied with $\sqrt{n}$ behaves asymptotically like a centered normal distribution with  variance equal to $2\pi f_0(0)$. Thus, if one chooses the expectation of the posterior as estimator, then this estimator is asymptotically equal to the sample mean. It is well known, see e.g.\ Theorems 7.1.1 in \cite{brockwell1991},  that the sample mean of Gaussian time series, centered at the true expectation and multiplied by $\sqrt{n}$, is asymptotically normal with variance $2\pi f_0(0)$, indeed the matrix entry in the first row and first column  of $\matrt \Sigma_n$ is given by
\begin{align*}
	&n\, \left( \gamma(0)+2\sum_{l=1}^{n-1}\frac{n-l}{n}\, \gamma(l) \right)=n (2\pi f_0(0)+o(1)),
\end{align*}
such that $\matr V_{0;n}$ as in \eqref{eq_cov_true} multiplied with $n$ coincides asymptotically with $2\pi f_0(0)$.

Consequently, the corresponding Bayesian inference is second-order correct in the above sense and the corresponding credibility sets  are also asymptotical confidence intervals. 
\end{Example}
\setcounter{Example}{\value{temp}}

\setcounter{temp}{\value{Example}}
\setcounter{Example}{\value{ex_linear}}

\begin{Example}[Time series with unknown linear trend (cont.)]
	In this case, the frequency domain signal is (in general) nonzero in all components, c.f. Example (c) in Section~8.1 in~\cite{meier}. In particular, unlike in Example~\ref{ex_mean_model}, the asymptotic distribution now depends on all frequencies and not only frequency zero.
 In the one-parameter model with a linear trend but no intercept, i.e.\ for $\matr X_n=(1/n,2/n,\ldots,1)^T$, we get as in  Example (c) in Section~8.1 in~\cite{meier} for $n$ odd
\begin{align*}\matrt X_n^T=\frac{1}{\sqrt{2n}}((n+1)/\sqrt{2}, 1,\cot(\lambda_1/2),1,\cot(\lambda_2/2),\ldots, 1,\cot(\lambda_N/2))
\end{align*}
with $\lambda_k=2\pi k/n$ and $N=\lfloor (n-1)/2\rfloor$ as before. For $n$ even an additional term of $1/\sqrt{4n}$ appears at the end of the vector.

If the entries of $\matrt \Sigma_n-\matr D_{n,0}$ are uniformly bounded by $1/n$, which holds  under suitable assumptions, see e.g.\ Theorem 4.9 in~\cite{meier}, and by the boundedness assumptions on the spectral density, then \eqref{eq_asym_equiv} is fulfilled if  $ \|\matrt X_n\|_1=o(n)$.

In the above case, clearly, the sum over all entries of $\matrt X_n$ except for the cotangent-terms is of order $O(\sqrt{n})=o(n)$. Concerning the sum over the absolute values of the remaining cotangent-terms, note, that by the proof of Lemma 4.4.2 in \cite{KirchDiss}
 \begin{align*}
     |\cot(\omega_j/2)|\le \frac{1}{\sin(\pi j/n)}=O\left(\max\left(\frac n j,\frac{n}{n-j}\right)\right)
 \end{align*}
 and the sum over these terms is bounded by $O(n\log n)$. Due to the 
 additional $1/\sqrt{n}$ factor the sum over the cotangent-terms of $\matrt X_n$ is of order $O(\sqrt{n}\log n)=o(n)$. Consequently, the desired condition $ \|\matrt X_n\|_1=o(n)$ holds.
\end{Example}
\setcounter{Example}{\value{temp}}

While in the above two examples, the second-order correctness of the Bayesian procedure with misspecified Whittle-likelihood was shown, this is by no means guaranteed as the following example shows:

\begin{Example}[Time series with design matrix not yielding nominal frequentist coverage]\label{counterexample_1}
	Here, we will give an example, where \eqref{eq_asym_equiv} does not hold. To do so, we  consider a Gaussian AR(1) time series which allows us to calculate the matrices $\matr F_n^T\matr D_{n,0}^{-1}\matr F_n$ and  $\matr F_n^T\matr D_{n,0}^{-1}\matr F_n\, \matr \Sigma_n\, \matr F_n^T\matr D_{n,0}^{-1}\matr F_n$, which makes checking \eqref{eq_asym_equiv} for specific design matrices very easy. In particular, this provides a counter example where second-order correctness does not hold and corresponding credibility sets are not asymptotic confidence sets.

	To elaborate, let the noise $\{e_t\}_{t\in\mathbb{N}}$ be a causal Gaussian AR(1) process satisfying $e_t=\alpha_0 e_{t-1}+\epsilon_t$, $\epsilon_t\sim N(0,\sigma_0^2)$ with $0<|\alpha_0|<1$, $\sigma_0>0$ and spectral density $f_{AR}$, hence $\matr \Sigma_n=\sigma_0^{2}(1-\alpha_0^2)^{-1}\left(\alpha_0^{|i-j|}\right)_{i,j=1,\cdots,n}\in\mathbb{R}^{n\times n}$, see e.g.\ Example~3.2.2 in combination with Theorem 3.2.1 in \cite{brockwell1991}. Furthermore, by e.g.\ Theorem~4.4.2 in \cite{brockwell1991},
	$f_{AR}^{-1}=f_{MA}$ is the spectral density of an MA(1) time series  with moving average parameter $-\alpha_0$ and variance $\sigma_0^{-2}$ of the innovations, hence for $n\ge 4$,
\begin{align*}
\matr F_n^T\matr D_{n,0}^{-1}\matr F_n
=\sigma_0^{-2}
\left(
\begin{matrix}
1+\alpha_0^2&-\alpha_0&0&0&\cdots&-\alpha_0\\
-\alpha_0&1+\alpha_0^2&-\alpha_0&0&\cdots&0\\
0&-\alpha_0&1+\alpha_0^2&-\alpha_0&\cdots&0\\\vdots&\ddots&\ddots&\ddots&\ddots&\vdots\\
0&0&\cdots&-\alpha_0&1+\alpha_0^2&-\alpha_0\\
-\alpha_0&0&\cdots&0&-\alpha_0&1+\alpha_0^2\\
\end{matrix}
\right)
\end{align*}
which is the correspondent circulant matrix, see e.g.\ (4.5.5) in \cite{brockwell1991} noting that $m(h)$ there is the autocovariance function of an MA(1) time series which is zero for $|h|>1$, such that the $\lambda_j$ there coincide with $f_{MA}=f_{AR}^{-1}$ at $\omega_j$ for $n\ge 2$. To facilitate notation we further assume that $\sigma_0=1$ in the following, where the results for general $\sigma_0$ follow easily as this is just a multiplicative constant.

According to equation (15) of \cite{Galbraith_and_Galbraith1974}, noting that $M_n$ there is $\matr \Sigma_n^{-1}$ here, $\matr F_n^T\matr D_{n,0}^{-1}\matr F_n$ and $\matr\Sigma_n^{-1}$ only differ at the four corner entries. To be more specific, we have
\begin{align}\label{circ_inverse}
    \matr F_n^T\matr D_{n,0}^{-1}\matr F_n = \matr\Sigma_n^{-1} + \matr A_n,
\end{align}
where $\matr A_n = \alpha_0^2\matr E_{(1,1);n} - \alpha_0\matr E_{(1,n);n} - \alpha_0\matr E_{(n,1);n} + \alpha_0^2\matr E_{(n,n);n}$ and $\matr E_{(i,j);n}$ is an $n\times n$ matrix whose only non-zero entry is the $(i,j)$-th entry with value 1.

Therefore, 
\begin{align*}
    &\left[\matr F_n^T\matr D_{n,0}^{-1}\matr F_n\, \matr \Sigma_n\right]\, \matr F_n^T\matr D_{n,0}^{-1}\matr F_n- \matr F_n^T\matr D_{n,0}^{-1}\matr F_n 
    =\left[\matr F_n^T\matr D_{n,0}^{-1}\matr F_n\, \matr \Sigma_n - \matr{\mbox{Id}_n}\right]\, \matr F_n^T\matr D_{n,0}^{-1}\matr F_n \\
    &=\left[\matr A_n\matr\Sigma_n\right]\, \matr F_n^T\matr D_{n,0}^{-1}\matr F_n
    =\matr A_n + \matr A_n \matr\Sigma_n\matr A_n \\
    &=2\, \frac{\alpha_0^2-\alpha_0^{n+2}}{1-\alpha_0^2}\left(\matr E_{(1,1);n}+\matr E_{(n,n);n}\right) - \frac{(\alpha_0+\alpha_0^3)(1-\alpha_0^{n})}{1-\alpha_0^{2}}\left(\matr E_{(1,n);n} + \matr E_{(n,1);n}\right).
\end{align*}
This shows that again the only non-zero entries of the matrix \newline$\left[\matr F_n^T\matr D_{n,0}^{-1}\matr F_n\, \matr \Sigma_n\right]\, \matr F_n^T\matr D_{n,0}^{-1}\matr F_n- \matr F_n^T\matr D_{n,0}^{-1}\matr F_n$ are located at the four corner entries.

 Consequently, for a one-parameter regression design $\matr X_n=(x_1,\ldots,x_n)^T$, it holds
 \begin{align*}
	 &\matr X_n^T \left(\matr F_n^T\matr D_{n,0}^{-1}\matr F_n\, \matr \Sigma_n\, \matr F_n^T\matr D_{n,0}^{-1}\matr F_n- \matr F_n^T\matr D_{n,0}^{-1}\matr F_n \right) \matr X_n\\
	 &=\frac{1-\alpha_0^n}{1-\alpha_0^2}\,\left( 2\alpha_0^2\left(x_1^2+x_n^2\right)-2 \alpha_0 (1+\alpha_0^2)\, x_1x_n \right).
 \end{align*}
 Choosing, e.g., $\matr X_n=(1,\cdots,1,1+\sqrt{n})^T$ fulfills $\matr X_n^T\matr X_n=2(n+\sqrt{n})$ such that \eqref{stand_ass_regr} and consequently Assumption~\ref{ass:designMatrix} are satisfied. 
However, 
clearly \eqref{eq_asym_equiv} is not fulfilled as
\begin{align*}
	&\frac 1 n \,\matr X_n^T \left(\matr F_n^T\matr D_{n,0}^{-1}\matr F_n\, \matr \Sigma_n\, \matr F_n^T\matr D_{n,0}^{-1}\matr F_n- \matr F_n^T\matr D_{n,0}^{-1}\matr F_n \right) \matr X_n\to \frac{2\alpha_0^2}{1-\alpha_0^2}\neq 0.
\end{align*}
\end{Example}
The obvious difference between the first two and the last example is the fact that there was a dominating term in $\matr X_n$ in the latter, i.e.\ it does not fulfill a Noether-type condition as known from the asymptotic theory on linear rank statistics, see e.g.\ \cite{hajek}. We conjecture, that this may be key to distinguishing these two situations even though the exact time series structure given in the last example only yields a counter-example if the first or last term are dominating.

These considerations are also consistent with the fact that frequentist results concerning the asymptotic normality of time series regression often require much stronger assumptions on the $\matr X_n$ than Assumption~\ref{ass:designMatrix}, see e.g. \cite{amemiya1973}. Indeed, the design matrix in Example~\ref{counterexample_1} satisfies Assumption~\ref{ass:designMatrix} but fails to fulfill Assumption 2 of \cite{amemiya1973}, which is related to a Noether-condition in a non-triangular setting.

\section{Simulation and Case Study}\label{sec:simulation}

\subsection{Mean plus Autoregressive Error Model}\label{subsec:mean}
To compare the effect of a correctly specified parametric error model, a misspecified parametric error model and the nonparametric Whittle-approach on the estimates of the linear model coefficients, let us consider the simple mean plus AR(1) error  model
\begin{equation}\label{eq:misspecified}
Z_t = \mu + Y_t,\qquad \text{with } Y_t= \rho Y_{t-1} + e_t,  \quad e_{t} \stackrel{iid}{\sim} N(0,\sigma^2)
\end{equation} 
for $t=1,\ldots,n$.
 We simulate $N=500$ time series of  model (\ref{eq:misspecified}) with  $\mu=1$, $\sigma^2=1$,  and $\rho=\pm 0.7$   for increasing time series lengths of  $n=128, 256, 512$ and fit a  nonparametric error model using the Bernstein-Dirichlet process prior (NP), a correctly specified mean plus AR(1) error model (AR) and a misspecified mean plus iid  error model (WN). We use noninformative Jeffreys' priors for 
 $\mu$ and $\sigma^2$ and a Uniform(-1,1) prior on $\rho$.
Table \ref{table:AR_nuisance_mean} compares the posterior mean, the mean squared error, coverage and length of 90\% credible intervals for $\mu$ for the correctly specified and the misspecified error model.
 
Not surprisingly, the coverage of the correctly specified parametric model is  closest to $90\%$ with the discrepancy being due to prior influence. The coverage of the misspecified parametric model, on the other hand, is well off target  with under $50\%$ in the case of positive autocorrelation, $100\%$ for negative autocorrelation  and without any  noteworthy improvement with increasing sample size. The discrepancy  of the misspecified model is explained by an under-estimation (over-estimation) of the spectral density at frequency $0$ for positive (negative) autocorrelation. The nonparametric model also has  a coverage that is  too low but it increases with increasing sample size as predicted by the theoretic results. Here, the discrepancy is due to a less precise estimator for the spectral density at frequency zero  in addition to the influence of the complicated prior  compared to the correctly specified parametric model. The undercoverage not only for positive but also negative autocorrelation is somewhat surprising but a closer look at the connection between the posterior of the spectral density at frequency zero and the posterior of the mean in the simulations reveals the usual connection of a tendency of a higher coverage conditioned on an overestimation of the spectral density at frequency $0$.

 \begin{table}[ht]
 \renewcommand{\arraystretch}{1.1}
 \label{table:AR_nuisance_mean}
\centering
\begin{tabular}{l|rrr|rrr|rrr}
  \hline
   & \multicolumn{3}{c}{$n=128$} &  \multicolumn{3}{c}{$n=256$} & \multicolumn{3}{c}{$n=512$}\\
 $\rho=0.7$  & NP & AR & WN & NP & AR & WN& NP & AR & WN \\
  \hline
   $\hat{\mu}$ & 0.9918 & 0.9919 & 0.9918 & 1.0035 & 1.0033 & 1.0034 & 1.0008 & 1.0007 & 1.0008 \\
   MSE & 0.3185 & 0.3207 & 0.3186 & 0.2176 & 0.2188 & 0.2177 & 0.1528 & 0.1541 & 0.1529 \\
  Coverage & 0.6860 & 0.8700 & 0.4660 & 0.7240 & 0.8940 & 0.4860 & 0.7300 & 0.8680 & 0.4960 \\
   Length & 0.6769 & 1.0283 & 0.4048 & 0.4917 & 0.7021 & 0.2881 & 0.3514 & 0.4851 & 0.2033 \\
   \hline
& & & & & & & & & \\
 $\rho=-0.7$  & & & & & & & & & \\  
  \hline
   $\hat{\mu}$ & 0.9986 & 0.9985 & 0.9988 & 1.0005 & 1.0006 & 1.0006 & 1.0000 & 1.0001 & 1.0001 \\
   MSE & 0.0571 & 0.0566 & 0.0572 & 0.0387 & 0.0387 & 0.0387 & 0.0271 & 0.0272 & 0.0272 \\
   Coverage  & 0.6660 & 0.8780 & 1.0000 & 0.7380 & 0.8820 & 1.0000 & 0.7740 & 0.8780 & 1.0000 \\
   Length & 0.1258 & 0.1726 & 0.4081 & 0.0971 & 0.1210 & 0.2897 & 0.0687 & 0.0849 & 0.2044 \\
\hline
   \hline
\end{tabular}
\caption{This table compares the posterior mean, the mean squared error, coverage and length of a 90\% posterior credible interval for $\mu$ for the correctly specified and misspecified error model.}
\end{table}

 \subsection{Linear Model with Autoregressive Errors}
 The main objective of this simulation study is to confirm empirically that the
 marginal posterior distribution of the linear model coefficients contracts at parametric rate. To this end, we simulate $N=500$ time series from a simple linear regression
plus AR(1) error model:
 \begin{equation}
 Z_t= \beta_0 + \beta_1 x_t + Y_t,\qquad\text{with }Y_t =\rho Y_{t-1} + e_t,  \quad e_{t} \stackrel{iid}{\sim} N(0,\sigma^2)
 \end{equation}
 with parameters $\beta_0=1$, $\beta_1=100$, for the linear model coefficients, $x_t=\frac{t}{n}$, $t=1,\ldots,n$, and $\sigma=1$ for increasing lengths of the time series $n=128,256,512$. The values for the  autoregressive coefficients of the AR(1) errors are chosen as in Section  \ref{subsec:mean}.
 
 Again, we fit a correctly specified linear regression plus AR(1) error model, a misspecified linear regression plus iid  error model and a nonparametric error model using the Bernstein-Dirichlet process prior. We use noninformative Jeffreys' priors for 
 $\beta_0$, $\beta_1$ and $\sigma^2$ and a Uniform(-1,1) prior on $\rho$.
Table \ref{table:ARnuisance-lm} compares the posterior means of the linear regression coefficients, their mean squared error (MSE) and the frequentist coverage and average length of a 90\% posterior credible interval.
\begin{table}[ht]
\label{table:ARnuisance-lm}
\centering
\begin{tabular}{l|rrr|rrr|rrr}
  \hline
   & \multicolumn{3}{c}{$n=128$} &  \multicolumn{3}{c}{$n=256$} & \multicolumn{3}{c}{$n=512$}\\
 $\rho=0.7$   & NP & AR & WN & NP & AR & WN& NP & AR & WN \\
  \hline
 $\hat{\beta}_0$& 0.9951 & 0.9985 & 0.9946 & 1.0010 & 0.9991 & 0.9979 & 0.9830 & 0.9844 & 0.9850 \\
  MSE($\beta_0$) & 0.0049 & 0.0015 & 0.0054 & 0.0010 & 0.0009 & 0.0021 & 0.0170 & 0.0156 & 0.0150 \\
Coverage($\beta_0$) & 0.6420 & 0.7680 & 0.4780 & 0.6740 & 0.7680 & 0.4840 & 0.7440 & 0.8220 & 0.4940 \\
  Length($\beta_0$)  & 1.2118 & 1.4655 & 0.7892 & 0.9357 & 1.0756 & 0.5665 & 0.6994 & 0.7777 & 0.4041 \\
    $\hat{\beta}_1 \times 10^{-2}$ & 1.0036 & 0.9999 & 1.0045 & 0.9998 & 0.9998 & 0.9998 & 1.0203 & 1.0176 & 1.0161 \\
   MSE($\beta_1$)  & 0.0036 & 0.0031 & 0.0045 & 0.0218 & 0.0176 & 0.0158 & 0.0203 & 0.0176 & 0.0161 \\
   Coverage($\beta_1$) & 0.6560 & 0.7860 & 0.4720 & 0.6640 & 0.7820 & 0.4520 & 0.7520 & 0.8120 & 0.5100 \\
   Length($\beta_1$)& 2.1525 & 2.5043 & 1.3589 & 1.6940 & 1.8847 & 0.9784 & 1.2791 & 1.3775 & 0.6989 \\
   \hline
    & & & & & & & & & \\
 $\rho=-0.7$  &  &  &  &  & & &  & &  \\
  \hline
   $\hat{\beta}_0$ & 0.9977 & 0.9996 & 0.9985 & 1.0000 & 0.9998 & 0.9990 & 0.9975 & 0.9975&  0.9968 \\
  MSE($\beta_0$) & 0.0023 & 0.0004 & 0.0015 & 0.0000 & 0.0002 & 0.0010 & 0.0025 & 0.0025 & 0.0032 \\
   Coverage($\beta_0$) & 0.6960 & 0.7960 & 1.0000 & 0.7940 & 0.8000 & 1.0000 & 0.8020 & 0.8320 & 1.0000 \\
   Length($\beta_0$)  & 0.2501 & 0.2893 & 0.8307 & 0.2038 & 0.2018 & 0.5789 & 0.1431 & 0.1417 & 0.4080 \\
    $\hat{\beta}_1 \times 10^{-2}$ & 1.0019 & 0.9999 & 1.0004 & 0.9999 & 0.9999 & 0.9999 & 1.0025 & 1.0027 & 1.0039 \\
   MSE($\beta_1$) & 0.0019 & 0.0011 & 0.0004 & 0.0039 & 0.0033 & 0.0021 & 0.0025 & 0.0027 & 0.0039 \\
    Coverage($\beta_1$) & 0.7180 & 0.8360 & 1.0000 & 0.8200 & 0.8380 & 1.0000 & 0.8260 & 0.8320 & 1.0000 \\
    Length($\beta_1$)& 0.4629 & 0.5159 & 1.4305 & 0.3765 & 0.3608 & 0.9997 & 0.2646 & 0.2536 & 0.7057 \\
\hline
\end{tabular}
\caption{This table compares the posterior mean, the mean squared error, coverage and length of a 90\% posterior credible interval of the regression coefficients for the correctly specified and misspecified error model.}
\end{table}

Similar to the previous case in Section~\ref{subsec:mean}, the coverage of both intercept and slope of the correctly specified parametric model is  closest to $90\%$. The coverage of the misspecified parametric model, on the other hand, is well below $90\%$ target  with under $50\%$ in the case of positive autocorrelation, $100\%$ for negative autocorrelation. The coverage does not improve with increasing sample size. As discussed before, this can be explained by an under-estimation (over-estimation) of the spectral density at frequency $0$ for positive (negative) autocorrelation. While the nonparametric model also has a coverage that is  too low, it increases with increasing sample size as predicted by the theoretical results. In this case, the undercoverage arises from the use of a less precise estimator for the spectral density at frequency zero, along with the impact of the more complex prior, as opposed to the correctly specified parametric model.

 \subsection{Case Study}

\begin{figure}
  \centering
  \includegraphics[width=1\columnwidth, keepaspectratio]{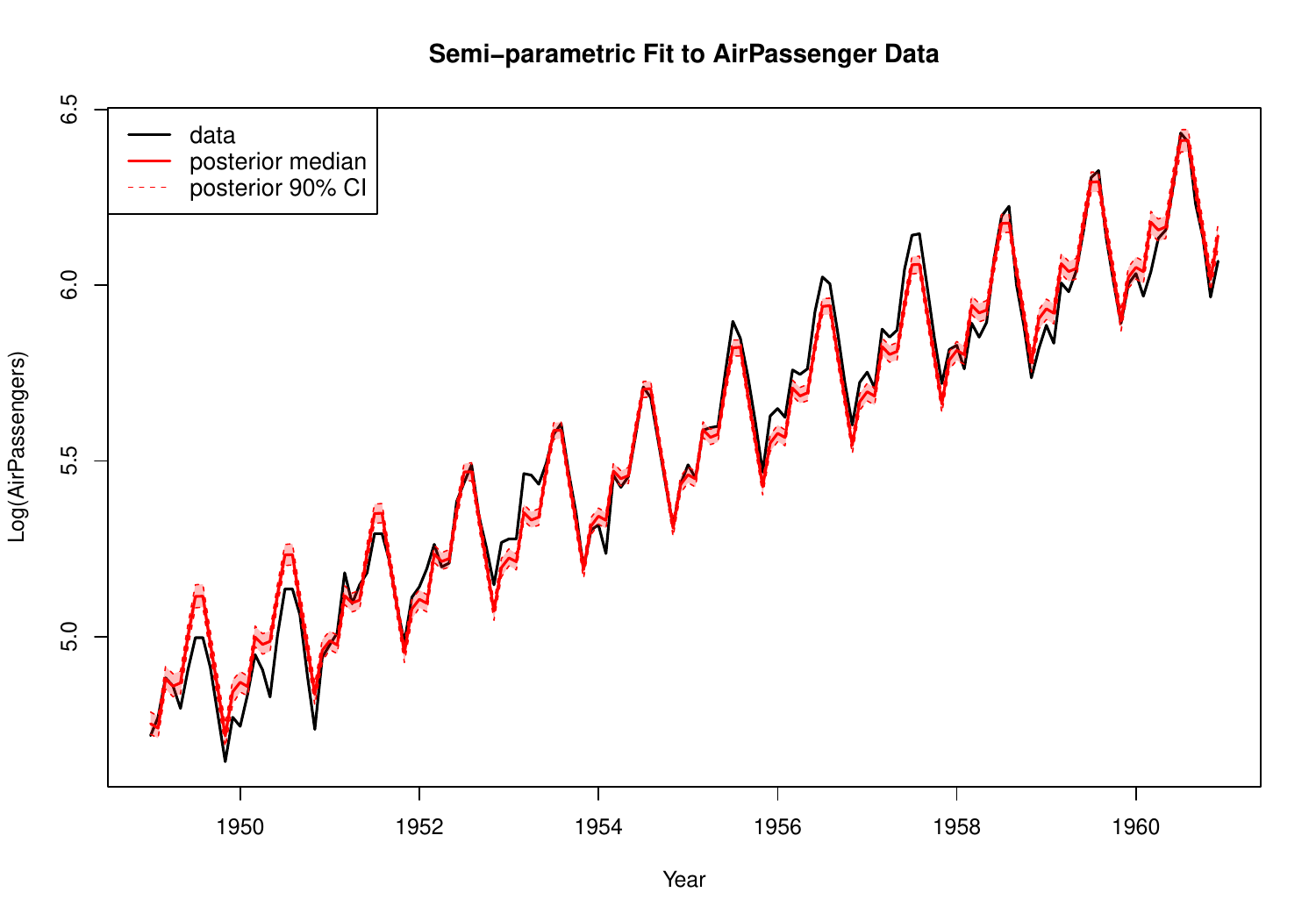}
  \captionof{figure}{
  Number of monthly airline passengers on log scale in black, posterior median of 6000 MCMC samples in red and 90\% posterior credible intervals in transparent red between 5 and 95\% quantiles in dashed red.
}
  \label{fig:AirPassFit}
\end{figure}

Here we analyze the “AirPassengers” dataset in R that contains the monthly total number (in thousands) of international airline passengers from 1949 to 1960. The data shows a clear linear temporal trend as well as seasonality.  Using Whittle's likelihood, we fit a linear model to the logarithmic data, denoted by $Z_t$, using standardised time and dummy 0-1 variables for each month as covariates while omitting the intercept (in contrast to the parametrization in Example~\ref{example:sesonality} of Section~\ref{subsec:definition} that contains the intercept but not the last dummy variable), i.e.,
 \[ 
Z_t = \beta_1 x^{(1)}_t + \sum_{i=2}^{13} \beta_i x_t^{(i)} +\epsilon_t
\]
where $x^{(1)}_t=\frac{t}{n},\; x^{(i)}_t=1$ if $t$ is a multiple of $i$ and 0 otherwise, for $t=1,\ldots, n=144$ and $i=2,\ldots 13$. The error process $\epsilon_t$ is assumed to be a mean zero Gaussian time series with spectral density $f(\omega)$. We use a flat Normal prior with mean 0 and standard deviation of 10,000 for $\beta_1$, flat Normal priors with mean 0 and standard deviation 100 for $\beta_i$, $i=2,\ldots, 13,$ and the Bernstein-Gamma process prior for the error process. In the univariate case, a normalized version of the Bernstein-Hpd matrix-Gamma process prior is the Bernstein-Dirichlet process, i.e.,
\begin{align*}
  f(\omega)= f(\omega|\Phi, k) := \tau \sum_{j=1}^k\Phi'(I_{j,k})b(\omega/\pi|j,k-j+1), \quad 0 \leq \omega \leq \pi, \\
  \Phi \sim \operatorname{DP}(M,G_0), \quad k \sim p(k),\quad \tau\sim \Gamma(\alpha,\beta)
\end{align*}
where $\operatorname{DP}(M,G_0)$ denotes the Dirichlet process with concentration parameter $M$ and base measure $G_0$. Thus, we use the Bernstein-Dirichlet prior for the spectral density of the error process, with hyperparameters set to $G_0=\mbox{Uniform}[0,1]$, $M=1$, $p(k) \propto \exp( -0.01 k^2)$, and $\alpha=\beta=0.001$.

We ran 50000 MCMC iterations with a burnin period of 20,000 and a thinning factor of 5. Figure \ref{fig:AirPassFit} shows the data overlaid by the posterior median and 90\% posterior credible intervals for trend and seasonality.
Both increasing linear trend over time as well as the seasonality components are accurately estimated while taking the serially correlated errors into account when calculating the marginal posterior credible bands.

\section{Proof of Theorem \ref{eq:bvmLinearModel}}\label{sec:proof}

The following theorem constitutes the main tool for the proof of Theorem~\ref{eq:bvmLinearModel}. It is due to~\cite{ghosal2017fundamentals}.
A general semiparametric Bernstein-von-Mises theorem has also been presented in \cite{bickel2012semiparametric}.
However, the authors there restrict their attention to the iid setting.
One possible route for an alternative proof of Theorem~\ref{eq:bvmLinearModel} could be 
a generalization of their results (and the proof technique, which is closely related
to concentration of partially misspecified posterior distributions as discussed 
in~\cite{kleijn2012bernstein}) to the non-iid setting.
\begin{Theorem}[General semiparametric BvM]\label{th:bvmMultiGeneral}
	Let~$Z_1,\ldots,Z_n$ be random variables  with likelihood
function having Lebesgue density~$p^n(\vect Z_n|\vect\theta,\eta)$, where~$\vect\theta\in\Theta\subset\mathbb R^r$
and~$\eta\in\mathcal H$ and~$\mathcal H$ being some measurable space.
Let the prior on~$(\vect\theta,\eta)$ be of the form~$P(d\vect\theta,d\eta)=P(d\vect\theta)P(d\eta)$.
Assume that there exist sequences of subsets~$\mathcal H_n\subset\mathcal H$ and
\begin{align}\label{eq:ThetaNHNDefT}
	\Theta_n=
	\left\{ \vect\theta\in\Theta\colon(\vect\theta-\vect\theta_0)^T\matr S_n(\vect\theta-\vect\theta_0)\leq \sqrt{\lambda_{\min}(\matr S_n)} 
	\right\}
\end{align}
with $\matr S_n$ as in (b). 
Additionally, let the following properties hold:
\begin{enumerate}[label=(\alph*)]
\item Prior Positivity: In an open neighbourhood of~$\vect\theta_0$, the prior~$P(d\vect\theta)$
possesses a continuous and strictly positive Lebesgue density.
\item Local Asymptotic Normality (LAN): There exists a tight sequence of random
variables~$(\vect G_n)$ (with respect to $P_0^n$, where~$P_0^n$ denotes the distribution of~$\vect Z_n$ under~$(\vect\theta_0,\eta_0)$), a sequence of spd matrices~$(\matr S_n)$ with $\lambda_{\min}(\matr S_n)\to \infty$ such that
\[
  \log \frac{p^n(\vect Z_n|\vect\theta,\eta)}{p^n(\vect Z_n|\vect\theta_0,\eta)}
  = \left( \vect\theta-\vect\theta_0 \right)^T\matr S_n^{1/2}\vect G_n
  -\frac{1}{2}\left( \vect\theta-\vect\theta_0 \right)^T\matr S_n\left( \vect\theta-\vect\theta_0 \right)
  + R_n(\vect\theta,\eta)
\]
holds for~$(\vect\theta,\eta)\in\Theta\times\mathcal H$ with  remainder term fulfilling, as $n \to \infty$, in $P_0^n$-probability
\begin{align}\label{eq:LANwhatwastoShow}
  \sup_{(\vect\theta,\eta)\in\Theta_n\times\mathcal H_n}\frac{\left|R_n(\vect\theta,\eta)\right|}{1+(\vect\theta-\vect\theta_0)^T\matr S_n(\vect\theta-\vect\theta_0)}
  \to 0.
\end{align}
The distribution of~$(\vect G_n)$ as well as the quantities $(\matr S_n)$ may depend on~$(\vect\theta_0,\eta_0)$ 
but not on~$(\vect\theta,\eta)$.
\item Joint 
	Conditional Nuisance Posterior Concentration: It holds
\[ 
P^n\left( (\vect\theta,\eta)\in\Theta_n\times\mathcal H_n | \vect Z_n \right) \to 1 \quad \text{in } P_0^n \text{-probability}.
\]
\end{enumerate}
Then, the marginal posterior of~$\vect\theta$ fulfils
\[
\delta_{\operatorname{TV}}\left[ P^n(\vect\theta\in\cdot|\vect Z_n), N_r\left( \vect \theta_0+\matr S_n^{-1/2}\vect G_n,\matr S_n^{-1}  \right) \right] \to 0
\]
in $P_0$-probability.
\begin{proof}
The result is a modification of Theorem~12.8 in~\cite{ghosal2017fundamentals},
from scalar to vector-valued parameters of interest and without the usage of
so-called \emph{least-favourable transformations}. The beginning of the proof is analogous to their proof with the decomposition as in (b). The arguments leading to (12.9) simplify  due to the assumption of independent priors on $\vect \theta$ and $\eta$: Indeed, in their notation $\widetilde{Q}_n(\theta)=\int_{\mathcal{H}_n}p^n(\vect Z_n|\vect\theta_0,\eta)\,P(d\eta)$, which does only depend on $\theta_0$ but not on $\theta$, such that (12.9) follows immediately by~\eqref{eq:LANwhatwastoShow}. 
Next, $\Theta_n$ shrinks to $\vect\theta_0$ since by
$(\vect\theta-\vect\theta_0)^T\matr S_n(\vect\theta-\vect\theta_0)\geq\lambda_{\min}(\matr S_n)\, \|\vect\theta-\vect\theta_0\|^2$ it holds  $ \|\vect\theta-\vect\theta_0\|^2 \le 1/\sqrt{\lambda_{\min}(\matr S_n)}\to 0$ by the definition of $\Theta_n$ and by $\lambda_{\min}(\matr S_n)\to \infty$.
This in combination with Property (a) 
allows us to replace the integral with respect to the prior by the corresponding integral with respect to the Lebesgue measure as in the proof of Theorem~12.8 in~\cite{ghosal2017fundamentals}.
By definition of $\Theta_n$ the set  $\matr S_n^{1/2}\left( \Theta_n-\vect\theta_0 \right)$ is the closed ball around $0$ with radius $(\lambda_{\min}(\matr S_n))^{1/4}\to\infty$ such that $\matr S_n^{1/2}\left( \Theta_n-\vect\theta_0 \right)\to \mathbb{R}^r$. This allows to replace the neighborhood $\Theta_n$ by $\mathbb{R}^k$ in the remaining integrals which shows that the posterior probability of a set $B$ is asymptotically proportional to 
\begin{align*}
	&\int_B\exp\left( \left( \vect\theta-\vect\theta_0 \right)^T\matr S_n^{1/2}\vect G_n
  -\frac{1}{2}\left( \vect\theta-\vect\theta_0 \right)^T\matr S_n\,\left( \vect\theta-\vect\theta_0 \right)
  \right)\,d\vect\theta\\
  &\propto \int_B\exp\left( -\frac 1 2 (\vect\theta-\vect\theta_0- \matr S_n^{-1/2}\vect G_n)^T\matr S_n (\vect\theta-\vect\theta_0-\matr S_n^{-1/2}\vect G_n) \right) \,d\vect\theta,
\end{align*}
which yields the assertion.
\end{proof}
\end{Theorem}

\begin{Remark}\label{rem_BvM_misspecified}
	The proof does not require that the posterior is calculated with the correct likelihood as long as the tightness assumptions and posterior concentrations hold under the true likelihood. Thus, it allows to derive Bernstein-von-Mises theorems for misspecified situations.

	We do not make use of this observation in the  proof of Theorem~\ref{eq:bvmLinearModel} below but instead use the contiguity of Whittle's and the true likelihood. While the LAN condition can also be proved under the true likelihood along the lines of the proof of Theorem~\ref{lemma:LANresult} with only little more effort, we do require the contiguity in the proof of the contractions rates in Theorem~\ref{th:jointContr}. Nevertheless, in the future, the fact that Theorem~\ref{th:bvmMultiGeneral} allows for misspecification could be helpful to take the result of Theorem~\ref{eq:bvmLinearModel} beyond the Gaussian case. For the latter consistency of the spectral density estimator has been shown in \cite{TangYifu2023Pcft}  but the technique there does not give contraction rates as required here.
\end{Remark}

\subsubsection*{Proof of Theorem~\ref{eq:bvmLinearModel}}
We will show that the assumptions of the above theorem hold under Whittle's likelihood at $(\vect \theta_0,f_0)$ with  $\Theta:=\mathbb R^r$, $\mathcal H:=\mathcal C_{\tau_0,\tau_1}$ as in \eqref{eq:truncationSetForBvm} and~$\matr S_n=\matrt X_n^T\matr D_{n,0}^{-1}\matrt X_n$. 
By Lemma~\ref{lemma:cfLike} in the appendix, the definition of $\mathcal C_{\tau_0,\tau_1}$ and Assumption~\ref{ass:designMatrix}  it holds 
\begin{align}\label{eq_lambda_mintoinfty}
	&\lambda_{\min}(\matr S_n)\ge \lambda_{\min}(\matr D_{n,0}^{-1})\,\lambda_{\min}(\matr X_n^T\matr X_n)\ge \tau_1^{-1}\,\lambda_{\min}(\matr X_n^T\matr X_n)\to \infty.
\end{align}

The independence assumption of the prior holds by the prior specification in Section~\ref{section_prior}, while the assumption of prior positivity (as in (a) of Theorem~\ref{th:bvmMultiGeneral}) holds by Assumption~\ref{ass:thetaPrior}.

Furthermore, let $\Theta_n$ as in \eqref{eq:ThetaNHNDefT} and define
\begin{align}
 &\mathcal H_n :=\left\{ f\in\mathcal H\colon \|f-f_0\|_\infty \leq 1/\log(n) \right\}.\label{eq:ThetaNHNDefH}
\end{align}

It remains to show the LAN property as in (b) of Theorem~\ref{th:bvmMultiGeneral} as well as the posterior concentrations as in (c) for the above choices (under $P_W(\cdot|\vect \theta_0,f_0)$). Both will be shown  in the following two subsections:
The LAN property  with $\vect G_n=\matr S_n^{-1/2}\matrt X_n^TD_{n,0}^{-1}(\vectt Z_n-\matrt X_n\vect\theta_0)$ is given in Theorem~\ref{lemma:LANresult} and the joint posterior concentration in
Theorem~\ref{th:jointContr}.

Finally, it holds \begin{align*}
	&\vect \theta_0+\matr S_n^{-1/2} \vect G_n=\matr S_n^{-1}\matrt X_n^TD_{n,0}^{-1}\vectt Z_n,
\end{align*}
such that the total variation distance as given in Theorem~\ref{eq:bvmLinearModel} converges to 0 in $P_{W}^n(\cdot|\vect \theta_0,f_0)$. By contiguity of Whittle's likelihood with the true likelihood (see e.g.\ \cite{choudhuriContiguity}) it then also converges in $P^n_0$.

\subsection{LAN property}
In this section, we will establish the LAN property as required in Theorem~\ref{th:bvmMultiGeneral} (b).
\begin{Theorem}\label{lemma:LANresult}
Under the assumption of Theorem~\ref{eq:bvmLinearModel}, it holds
\[
  \log\frac{p_W^n(\vect z_n|\vect\theta,f)}{p_W^n(\vect z_n|\vect\theta_0,f)}
  = (\vect\theta-\vect\theta_0)^T\matr S_n^{1/2}\vect G_n
    -\frac{1}{2}(\vect\theta-\vect\theta_0)^T\matr S_n(\vect\theta-\vect\theta_0)
    + R_n(\vect\theta,f),
\]
where~$\matr S_n=\matrt X_n^T\matr D_{n,0}^{-1}\matrt X_n$ and $\vect G_n=\matr S_n^{-1/2}\matrt X_n^T\matr D_{n,0}^{-1}(\vectt Z_n-\matrt X_n\vect\theta_0)$, which is a tight sequence of random vectors (under $P_W^n(\cdot|\vect\theta_0,f_0)$), and a remainder term~$R_n(\vect\theta,f)$ that fulfills~\eqref{eq:LANwhatwastoShow}
with~$\Theta_n$ from~\eqref{eq:ThetaNHNDefT} and $\mathcal H_n$ from  \eqref{eq:ThetaNHNDefH}. 
\end{Theorem}

\begin{proof}
	Denote~$\vectt Z_{n,0}:=\vectt Z_n-\matrt X_n\vect\theta_0$.
With a few calculations we obtain that the log likelihood ratio for Whittle's likelihood~\eqref{eq:whittleLinearModel}
is given as
\begin{align}
  &\log\frac{p_W^n(\vect z_n|\vect\theta,f)}{p_W^n(\vect z_n|\vect\theta_0,f)}= (\vect\theta-\vect\theta_0)^T\matrt X_n^T\matr D_n^{-1}\vectt Z_{n,0}
    -\frac{1}{2}(\vect\theta-\vect\theta_0)^T\matrt X_n^T\matr D_n^{-1}\matrt X_n(\vect\theta-\vect\theta_0)\label{eq:logLikRat0} \\
    &= (\vect\theta-\vect\theta_0)^T\matr S_n^{1/2} \vect G_n    -\frac{1}{2}(\vect\theta-\vect\theta_0)^T\matr S_n(\vect\theta-\vect\theta_0)
    + R_n(\vect\theta,f)\nonumber 
\end{align}
with remainder term
\begin{align*}
  &R_n(\vect\theta,f)\\
  &= (\vect\theta-\vect\theta_0)^T\matrt X_n^T\left(\matr D_{n}^{-1}-\matr D_{n,0}^{-1}\right)\vectt Z_{n,0}
    -\frac{1}{2}(\vect\theta-\vect\theta_0)^T\matrt X_n^T\left(\matr D_{n}^{-1}-\matr D_{n,0}^{-1}\right)\matrt X_n(\vect\theta-\vect\theta_0)\\
  &= R_{n,1}(\vect\theta,f)-\frac 1 2 \, R_{n,2}(\vect\theta,f).
\end{align*}
Under $P_{W;0}^n(\cdot)=P_W^n(\cdot|\vect\theta_0,\eta_0)$
it holds~$\vectt Z_{n,0}\sim N_n(\vect 0,\matr D_{n,0})$
with diagonal frequency domain covariance matrix~$\matr D_{n,0}=\matr D_n(f_0)$.
Consequently, $\vect G_n=\matr S_n^{-1/2}\matrt X_n^T\matr D_{n,0}^{-1}\vectt Z_{n,0}$ is also normally distributed with mean zero and the ($r$-dimensional) identity matrix as covariance, hence tight.
The Cauchy-Schwarz  inequality yields
\begin{align*}
|R_{n,1}(\vect\theta,f)|\leq \sqrt{ (\vect\theta-\vect\theta_0)^T \matr S_n (\vect\theta-\vect\theta_0)}\left\|\matr S_n^{-1/2}\matrt X_n^T\left(\matr D_{n}^{-1}-\matr D_{n,0}^{-1}\right)\vectt Z_{n,0}\right\|.
\end{align*}
Together with the fact that $x\mapsto \sqrt{x}/(1+x)$, $x\ge 0$, is a bounded function, we can show that $R_{n,1}$ satisfies~\eqref{eq:LANwhatwastoShow} by proving $\sup_{f\in\mathcal{H}_{n}}\|\matr S_n^{-1/2}\matrt X_n^T\left(\matr D_{n}^{-1}-\matr D_{n,0}^{-1}\right)\vectt Z_{n,0}\|\to0$ in $P_{W;0}^n$-probability. To this end, we employ Lemma~\ref{lemma:uniform_convergence} which requires two conditions. Let $\matr W_n(f)=\matr S_n^{-1/2}\matrt X_n^T\matr D_{n}^{-1}(f)$. Then for any $f,g\in\mathcal{H}_n$,
\begin{align*}
&\|W_n(f)-W_n(g)\|_{F}=\sqrt{\tr\left(\matr S_n^{-1/2}\matrt X_n^T\left(\matr D_{n}^{-1}(f)-\matr D_{n}^{-1}(g)\right)^2\matrt X_n\matr S_n^{-1/2}\right)} \\
&= O(1) \|f-g\|_{\infty}\sqrt{\tr\left(\matr S_n^{-1/2}\matrt X_n^T\matr D_{n,0}^{-1}\matrt X_n\matr S_n^{-1/2}\right)}
=O(1) \|f-g\|_{\infty},
\end{align*}
due to an application of Lemma~\ref{lemma:cfLike} (c) and noting that $\matr S_n^{-1/2}\matrt X_n^T\matr D_{n,0}^{-1}\matrt X_n\matr S_n^{-1/2}=\mbox{Id}_r$, where the $O(1)$ terms do not depend on $f,g$. The above derivation shows both conditions (i) and (ii) of Lemma~\ref{lemma:uniform_convergence}, establishing \eqref{eq:LANwhatwastoShow} for $R_{n,1}(\vect\theta,f)$. 
Similarly,
\begin{align*}
 & R_{n,2}(\vect\theta,f)
  = (\vect\theta-\vect\theta_0)^T\matrt X_n^T(\matr D_n^{-1}-\matr D_{n,0}^{-1})\matrt X_n(\vect\theta-\vect\theta_0)
  \leq \left\|\matr D_n^{-1}-\matr D_{n,0}^{-1}\right\| \left\| \matrt X_n(\vect\theta-\vect\theta_0)  \right\|^2\\
  &=o(1)\, (\vect\theta-\vect\theta_0)^T \matr S_n (\vect\theta-\vect\theta_0).
\end{align*}
This establishes  \eqref{eq:LANwhatwastoShow} for $R_{n,2}(\vect\theta,f)$ because $x\mapsto x/(1+x)$, $x\ge 0$, is also a bounded function.
\end{proof}
\subsection{Joint posterior contraction}

In this section, we will show that~$\Theta_n\times\mathcal H_n$ from~\eqref{eq:ThetaNHNDefT} and \eqref{eq:ThetaNHNDefH}
are concentration sets for the joint posterior~$P_{W;\tau_0,\tau_1}^n$
from Theorem~\ref{eq:bvmLinearModel},
which is formulated as the main result in Theorem~\ref{th:jointContr}.
The proof consists of showing that~$\varepsilon_n$ from~\eqref{eq:varepsilonn} is a joint contraction rate
in the Hellinger topology (Section~\ref{sec:hellingerContr})
and that the marginal posterior of~$\vect\theta$ contracts at parametric rate 
(Section~\ref{sec:marginalParametric}).

\subsubsection{Hellinger contraction}\label{sec:hellingerContr}
In this section we prove the joint contraction rates in terms of the root average squared Hellinger distance~$d_{n,H}$, which
is defined as
\begin{equation}\label{eq:hellingerDefSemi}
  d_{n,H}^2\big((\vect\theta_0,f_0),(\vect\theta,f)\big)
  :=\frac 1 n\sum_{j=1}^nd_H^2\big( p_j(\cdot|\vect\theta_0,f_0),p_j(\cdot|\vect\theta,f) \big),
\end{equation}
where~$2\,d_H^2(p,q)=\int\left( \sqrt{p(z)}-\sqrt{q(z)} \right)^2\,dz$, i.e.\ $d_H(p,q)$ denotes the Hellinger distance between two probability densities~$p,q$, and~$p_j(\cdot|\vect\theta,f)$ denoting the probability density of
the~$j$-th Fourier coefficient~$\tilde Z_j=(\matr F_n\vect Z_n)_j$ under~$P_W^n(\cdot|\vect\theta,f)$.
Before we state the main theorem of this section we first derive a lower bound for the prior probability of Kullback-Leibler-type neighborhoods (under the Whittle likelihood) of the true parameters.

Denote by~$\E_{W;0}$ and~$\Var_{W;0}$ the mean and variance respectively
under the Whittle-likelihood ~$P_{W;0}^n$, consider the KL divergence
and associated variance term
\begin{align*}
  K_n\big( (\vect\theta_0,f_0),(\vect\theta,f) \big)
  &:= \frac{1}{n}\E_{W;0}\left[\log\frac{p_W^n(\vectt Z_n|\vect\theta_0,f_0)}{p_W^n(\vectt Z_n|\vect\theta,f)}\right]\\
  V_n\big( (\vect\theta_0,f_0),(\vect\theta,f) \big)
  &:= \frac{1}{n}\Var_{W;0}\left[\log\frac{p_W^n(\vectt Z_n|\vect\theta_0,f_0)}{p_W^n(\vectt Z_n|\vect\theta,f)}\right]
\end{align*}
and, for~$\varepsilon>0$, the KL-type neighborhood
\begin{equation}\label{eq:klTypeNeighborhood}
  B_n(\vect\theta_0,f_0,\varepsilon)
  := \left\{ (\vect\theta,f)\in\mathcal A_n\colon K_n\big( (\vect\theta_0,f_0),(\vect\theta,f) \big)<\varepsilon^2,
  V_n\big( (\vect\theta_0,f_0),(\vect\theta,f) \big)<\varepsilon^2 \right\}.
\end{equation}

\begin{Lemma}\label{lemma:prior:neighb}
	Under the assumptions of Theorem~\ref{eq:bvmLinearModel}, it holds for any sequence $\varepsilon_n=\rho n^{-s}$ with $0<s<\frac{a-1}{2a}$, $a$ as in Assumption~\ref{ass:fModel} and $\rho>0$
	\begin{equation}\label{eq:priorMassThetaF}
  P_{\tau_0,\tau_1}\big( B_n(\vect\theta_0,f_0,\varepsilon_n) \big)
  \geq \exp(-cn\varepsilon_n^2)
\end{equation}
 for a positive constant~$c$ for~$n$ large enough.
\end{Lemma}
\begin{proof}
For $\|f_1-f_0\|_{\infty}\le \tau_0$ it holds that $f_0(\omega)/f_1(\omega)\ge 1/2$, such that by Lemma~\ref{lemma:KLhelpRealValued}
\begin{align*}
 & nK_n\big( (\vect\theta_0,f_0),(\vect\theta,f) \big)
  \lleq \| \matr D_n-\matr D_{n,0} \|_F^2 + \|\matrt X_n(\vect\theta-\vect\theta_0)\|^2\\
 & \lleq  n\|f-f_0\|_\infty^2 + \|\matr X_n\|^2\, \|\vect\theta-\vect\theta_0\|^2
\end{align*}
with~$\matr D_{n,0}=\matr D_n[f_0]$ as before and an analogous assertion for $V_n\big( (\vect\theta_0,f_0),(\vect\theta,f) \big)$,
where the proportionality constants do not depend on~$n$.
Consequently, for some suitable constant $c>0$ (and $B_{\alpha}(\vect\theta_0)=\{\vect\theta: \|\vect\theta-\vect\theta_0\|\leq \alpha\}$, $B^{\infty}_{\alpha}(f_0)=\{f: \|f-f_0\|_{\infty}\leq\alpha\}$) \begin{align*}
	&B_{c\varepsilon_n\, \frac{\sqrt{n}}{\|\matr X_n\|}}(\vect\theta_0)\times B^{\infty}_{c\varepsilon_n}(f_0)\subset B_n(\vect\theta_0,f_0,\varepsilon_n).
\end{align*}
By the independence assumption lower bounds on each of the two balls also give a lower bound of
the prior mass in \eqref{eq:priorMassThetaF}. 
By Assumption~\ref{ass:thetaPrior} it holds
\begin{align*}
	&P\left( B_{c\varepsilon_n\, \frac{\sqrt{n}}{\|\matr X_n\|}}(\vect\theta_0) \right)\geq
	P\left( \left\{\vect\theta: \|\vect\theta - \vect\theta_0\|_{\infty}\leq c \varepsilon_n\, \frac{1}{\|\matr X_n\|} \right\} \right)\ggeq	\frac{\varepsilon_n^r}{\|\matr X_n\|^r}\\
	&\ggeq \exp\left( r \log(\varepsilon_n)-r \log(\|\matr X_n\|) \right).
\end{align*}
Because $\varepsilon_n\to 0$ with polynomial (in $n$) rate and $n\varepsilon_n^2\to \infty$ also with polynomial (in $n$) rate, it holds $\log(\varepsilon_n)\ggeq - n\varepsilon_n^2$. Furthermore, by Assumption~\ref{ass:designMatrix}, it holds $\log\left( \|\matr X_n\| \right)=o(n\varepsilon_n^2)$ showing that  for some suitable constant $c_1>0$
\begin{align*}
	P\left( B^{\infty}_{c\varepsilon_n\, \frac{\sqrt{n}}{\|\matr X_n\|}}(\vect\theta_0)\right)\geq\exp(-c_1n\varepsilon_n^2).
\end{align*}
Next, we show that 
\begin{align}\label{priormass}
	P_{\tau_0,\tau_1}(B^{\infty}_{c\varepsilon_n}(f_0))\geq\exp(-c_2n\varepsilon_n^2)
\end{align}
for some $c_2>0$, thus completing the proof of \eqref{eq:priorMassThetaF}.
Because the sup-norm is bounded by the Lipschitz-norm and because $f\in C_{\tau_0,\tau_1}$ for any $f$ close enough (in the Lipschitz-norm) to $f_0$ as in Theorem~\ref{eq:bvmLinearModel}, it holds (with $B^{L}_{\alpha}(f_0)=\{f: \|f-f_0\|_{L}\leq\alpha\}$) for $n$ large enough
\begin{align*}
	&P_{\tau_0,\tau_1}(B^{\infty}_{c\varepsilon_n}(f_0))\geq P_{\tau_0,\tau_1}(B^{L}_{c\varepsilon_n}(f_0))\geq P(B^{L}_{c\varepsilon_n}(f_0)),
\end{align*}
and we continue to bound the right hand side.
As in the proof of Lemma~7.10 in \cite{meierPaper} we get for any $k_n$
\begin{align*}
	&\|f-f_0\|_L\le 3k_n^2\sum_{j=1}^{k_n}|\Phi(I_{j,k_n})-F_0(I_{j,k_n})|+\|f_{k_n}-f_0\|_L,
\end{align*}
where 
$f_{k}(\omega)=\sum_{j=1}^{k}F_0(I_{j,k}) b(\omega/\pi|j,k-j+1).$ While the convergence rate of the Bernstein polynomials $f_{k_n}(\omega)$ to the true function in the sup-norm is of order $k_n^{-a/2}$ (see Theorem 1.6.2 of \cite{lorentz2012bernstein}), the rate in the Lipschitz norm is of order $k_n^{-(a-1)/2}$ -- which is the rate of convergence of the derivative of the Bernstein polynomials to the derivative of the true function under the given assumptions. This follows from Theorem 1.6.1  of \cite{lorentz2012bernstein} in combination with a slight refinement of the arguments in the beginning of Section 1.8. Choosing $k_n=n\varepsilon_n^2/\log(n)$ yields
\begin{align*}
	&	&\|f_{k_n}-f_0\|_L\lleq k_n^{-(a-1)/2}=\rho^{-a}\,\varepsilon_n \log(n)^{(a-1)/2} n^{-(a-1-2as)/2},
\end{align*}
such that $\|f_{k_n}-f_0\|_L\le c\varepsilon_n/2$ for $n$ large enough by $a-1-2as>0$. Thus,
\begin{align*}
	&P(B^{L}_{c\varepsilon_n}(f_0))\ge P(k=k_n)\,P\left( \left\{\Phi: \sum_{j=1}^{k_n}|\Phi(I_{j,k_n})-F_0(I_{j,k_n})|\le \frac{c\,\varepsilon_n}{6\, k_n^2} \right\}\right).
\end{align*}
By Assumption~\ref{ass:kPrior} it holds $P(k=k_n)\geq A_1\exp(- \kappa_1\, k_n\log(k_n))\geq A_1\exp(-\tilde c_1 k_n\log n)$ for some suitable $\tilde c_1$, and by analogous arguments to Lemma~7.12 in \cite{meierPaper} we get for suitable constants $\tilde{C}_2,\tilde c_2>0$ (noting that $\varepsilon_n/k_n^3$ converges to $0$ polynomially in $n$)
\begin{align*}
	&P\left( \left\{\Phi: \sum_{j=1}^{k_n}|\Phi(I_{j,k_n})-F_0(I_{j,k_n})|\le \frac{c\,\varepsilon_n}{6\, k_n^2} \right\}\right)\geq \tilde{C}_2\exp(-\tilde{c}_2 k_n\log n).
\end{align*}
Consequently, for suitable constants $\tilde C_3,\tilde c_3>0$ 
\begin{align*}
	&P(B^{L}_{c\varepsilon_n}(f_0))\ge \tilde C_3 \exp(-\tilde c_3 k_n \log n),
\end{align*}
which yields \eqref{priormass} because by definition $k_n\log n=n\varepsilon_n^2$.
\end{proof}

We state the result of the theorem below with respect to  the Whittle-likelihood $P_{W;0}^n=P_W^n(\cdot|\vect\theta_0,f_0)$ as in  the proof of Theorem~\ref{eq:bvmLinearModel}. However,  the assertion also holds with respect to $P_0^n$-probability by contiguity. 

\begin{Theorem}\label{lemma:jointHellinger}
	Under the assumptions of Theorem~\ref{eq:bvmLinearModel}, it holds
\[
P_{W;\tau_0,\tau_1}^n\Big( \left\{ (\vect\theta,f) \colon d_{n,H}\big( (\vect\theta,f),(\vect\theta_0,f_0) \big) \geq n^{-s} \right\} | \vect Z_n \Big) \to 0
\]
in~$P_{W;0}^n$-probability as~$n\to\infty$ for any $0<s<\frac{a-1}{2a}$ with $a$ as in Assumption~\ref{ass:fModel}.
\end{Theorem}
\begin{proof}
The proof is based on a general contraction rate theorem for independent, non-identically
distributed random variables (Theorem 8.19 and 8.20 in \cite{ghosal2017fundamentals}). First, define the contraction rates by 
\begin{equation}\label{eq:varepsilonn}
	\varepsilon_n= n^{-s} \quad\text{with }0<s<\frac{a-1}{2a}.
\end{equation}
Consider the following sieve structure~$\mathcal A_n$ on the parameter space~$\mathcal A:=\Theta\times\mathcal H$:\begin{equation}\begin{split}\label{eq:sieveDefThetaF}
  \mathcal A_n &:= \mathcal A_n^{(1)} \times \mathcal A_n^{(2)} \\
  \mathcal A_n^{(1)} &:= \left\{ \vect\theta\in\mathbb R^r \colon \|\vect\theta\|<\exp(\rho_1n\varepsilon_n^2) \right\}, \\
  \mathcal A_n^{(2)} &:= \bigcup_{k=1}^{k_n}\left\{ \mathfrak B(k,\vect w) \colon \vect w \in \mathbb R_{\geq 0}^k \right\} \cap \mathcal C_{\tau_0,\tau_1}, \quad k_n=\rho_2 n^{1-2s}/\log(n) \end{split}\end{equation}
with the Bernstein polynomial expansion operator~$\mathfrak B(k,\vect w):=\sum_{j=1}^kw_jb(\cdot|j,k-j+1)$ 
and positive constants~$\rho_1,\rho_2$ to be specified later. Clearly, Assumption (i) of Theorem 8.19 in \cite{ghosal2017fundamentals} is fulfilled for these choices by Lemma~\ref{lemma:prior:neighb}.
Next, we will establish the entropy bound  for the sieve~$\mathcal A_n$
in the Hellinger topology:
\begin{equation}\label{eq:metricEntropyThetaF}
   \log N\left(\xi\varepsilon_n, \mathcal A_n ,d_{n,H}\right)\leq n\varepsilon_n^2
\end{equation}
for every~$\xi>0$,
where $N$ denotes the covering number.
This implies Assumption~(ii) of  Theorem 8.19 in \cite{ghosal2017fundamentals}.

First, by Lemma~\ref{lemma:hellingerBoundRealValuedSemiparametric}
\[
  d_H^2\big( p_j(\cdot|\vect\theta_0,f_0), p_j(\cdot|\vect\theta,f) \big)
  \lleq |f(\omega_j)-f_0(\omega_j)| + |(\matrt X_n(\vect\theta-\vect\theta_0))_j|
\]
for~$j=1,\ldots,n$, yielding (where $\|\cdot\|_1$ is the vector $1$-norm)
\[
  d_{n,H}^2\big((\vect\theta_0,f_0),(\vect\theta,f)\big)
  \lleq \|f-f_0\|_\infty + \frac{1}{n}\left\| \matrt X_n(\vect\theta-\vect\theta_0) \right\|_1
  \lleq \|f-f_0\|_\infty + \frac{1}{\sqrt{n}} \left\| \matrt X_n\right\|\,\left\|\vect\theta-\vect\theta_0 \right\|,
\]
because by the Cauchy-Schwarz inequality it holds $\|\vect x\|_1\leq \sqrt{n}\, \|\vect x\|$ for any vector $\vect x$ of dimension $n$.
This implies (see e.g.\ Lemma~B.32 in the appendix of~\cite{meier}) for some positive constant~$c$ 
\begin{equation}\label{eq:metricEntropyZwischenschritt}
  \log N(\xi\varepsilon_n,\mathcal A_n,d_{n,H})
  \leq \log N(c\xi^2n\varepsilon_n^2\| \matr X_n \|^{-1},\mathcal A_n^{(1)},\|\cdot\|)
  + \log N(c\xi^2\varepsilon_n^2,\mathcal A_n^{(2)},\|\cdot\|_\infty).
\end{equation}
An application of Proposition C.2 in \cite{ghosal2017fundamentals}  reveals -- recalling~$\mathcal A_n^{(1)}$ from~\eqref{eq:sieveDefThetaF} -- that
the first summand on the right hand side of~\eqref{eq:metricEntropyZwischenschritt}
is bounded from above by \[
  r\log 3+r\rho_1n\varepsilon_n^2-r\log(c\xi^2n\varepsilon_n^2)+r\log(\|\matr X_n\|)
\leq 2r\rho_1n\varepsilon_n^2
\]
for~$n$ large enough by Assumption~\ref{ass:designMatrix} and the fact that $n\varepsilon_n^2$ grows polynomially in $n$.
By choosing~$\rho_1$ small enough, this can be bounded from above by~$\frac{1}{2}n\varepsilon_n^2$.

By Lemma~B.4 in \cite{choudhuri} and because $\varepsilon_n$ falls polynomially in $n$, the second summand on the right hand side of~\eqref{eq:metricEntropyZwischenschritt}
is  $\lleq k_n\log(n)=\rho_2\,n\varepsilon_n^2$. Thus, the second summand is also bounded by ~$\frac{1}{2}n\varepsilon_n^2$ for~$\rho_2$ small enough.
This concludes~\eqref{eq:metricEntropyThetaF}.
\par

The existence of tests as required by Theorem~8.19 in \cite{ghosal2017fundamentals} follows e.g. from Lemma~2 in \cite{ghosal2007convergence}, such that all assumptions of that theorem are fulfilled showing that $\varepsilon_n$ is a contraction rate on the sieve $\mathcal{A}_n$. 
The sequences $M_n$ are not needed here because the assertion holds for any $s<(a-1)/2a$ (instead of an $-(a-1)/2a$-rate with appropriate Log-terms).

In order to complete the proof we will now show that the posterior on the complement of $\mathcal{A}_n$ is an asymptotic zero-set by using Theorem~8.20 in \cite{ghosal2017fundamentals}.
First, note that, 
\begin{align*}
	&\mathcal{A}_n^c=\left(\mathcal{A}_n^{(1)c}\times \mathcal{A}_n^{(2)}\right)\cup \left(\mathcal{A}_n^{(1)}\times \mathcal{A}_n^{(2)c}\right),
\end{align*}
and thus by prior independence
\begin{align}\label{eq_p_sets}
  P_{\tau_0,\tau_1}(\mathcal{A}_n^{c})
  \le P(\mathcal{A}_n^{(1)c}) + P_{\tau_0,\tau_1}(\mathcal{A}_n^{(2)c}).
\end{align}
By Assumption~\ref{ass:thetaPrior} and the Markov inequality
\[
P(\mathcal{A}_n^{(1)c})
= P(\|\vect \theta\|\ge \exp{\rho_1 n\epsilon_n^2})
\leq \exp(-\rho_1 \nu n \epsilon_n^2)\, \E\|\vect\theta\|^{\nu}.\]
Concerning the second probability in \eqref{eq_p_sets} by Assumption~\ref{ass:kPrior} it holds for some $\tilde{c}>0$ 
\[
P_{\tau_0,\tau_1}(\mathcal{A}_n^{(2)c}) \lleq \sum_{k>k_n}p(k)\lleq 
 \exp(-\tilde ck_n).
\]
Thus, for some  $\bar{c}>0$, $P_{\tau_0,\tau_1}\left(\mathcal{A}_n^c  \right)\lleq \exp(-\bar{c}n^{1-2s}/\log(n))$, such that for $\bar{\varepsilon}_n=n^{\bar{s}}$ with $s<\bar{s}<(a-1)/2a$ it holds by Lemma~\ref{lemma:prior:neighb}
\begin{align*}
	&\frac{P_{\tau_0,\tau_1}\left(\mathcal{A}_n^c  \right)}{ P_{\tau_0,\tau_1}\big( B_n(\vect\theta_0,f_0,\bar{\varepsilon}_n) \big) }\lleq \exp\left(-\bar{c}n^{1-2s}/\log(n)+cn^{1-2\bar{s}}  \right)=o\left( \exp(-2n\bar{\varepsilon}_n^2) \right).
\end{align*}
An application of Theorem~8.20 in \cite{ghosal2017fundamentals} completes the proof.\end{proof}

\subsubsection{Marginal contraction at parametric rate}\label{sec:marginalParametric}

In this section, we will establish concentration of the marginal posterior of~$\vect\theta$
at parametric rate.
We start with showing the result for the case of a flat prior of the form~$P(d\vect\theta)\propto 1$
in Lemma~\ref{lemma:marginalParametricFlat}
and extend this result to the setting of Theorem~\ref{eq:bvmLinearModel} in Lemma~\ref{lemma:marginalParametric}.

We state the result of the lemma below w.r.t.\ the Whittle-likelihood,  in line with the proof of Theorem~\ref{eq:bvmLinearModel} but note the assertion also holds in $P_0^n$-probability using the contiguity argument.
\begin{Lemma}\label{lemma:marginalParametricFlat}
Let~$\vect Z_n=\matr X_n\vect\theta_0+\vect e_n$ for~$\vect\theta_0\in\mathbb R^r$,
where the design matrix~$\matr X_n$ is known and fulfils Assumption~\ref{ass:designMatrix}
and~$\{e_t\}$ is a stationary mean zero Gaussian time series with
spectral density~$f_0$ fulfilling Assumption~\ref{ass:fModel}.
Denote the true distribution of~$\vect Z_n$ by~$P_0^n$.
Let~$\tau_0\in(0,\min_{\omega}f_0(\omega))$ and~$\tau_1\in(\max_{\omega}f_0(\omega),\infty)$.
Let the prior on~$(\vect\theta,f)$ be given as~$P_{\tau_0,\tau_1}(d\vect\theta,df):=P_{\tau_0,\tau_1}(df)$
with~$P_{\tau_0,\tau_1}$ as in~\eqref{eq:truncationForBvMSet}
with Bernstein-Gamma prior~$P(df)$, where the prior on~$\Phi$ fulfils 
Assumption~\ref{ass:GPBVM} and the prior on~$k$ fulfils Assumption~\ref{ass:kPrior}.
Then it holds
\[
\widehat{P}_{W;\tau_0,\tau_1}^n\left( \{\vect\theta\in\mathbb R^r\colon\|\matr X_n(\vect\theta-\vect\theta_0)\|\geq M_n\}\times\mathcal{H}_n|\vect Z_n \right) \to 0
\]
in~$P_{W;0}^n$-probability as~$n\to\infty$ for every sequence~$(M_n)$ of positive numbers with~$M_n\to\infty$, where $\widehat{P}_{W;\tau_0,\tau_1}^n$ indicates the posterior under the flat prior for $\vect\theta$ and $\mathcal{H}_n$ is defined in \eqref{eq:ThetaNHNDefH}.
\begin{proof}
Following the arguments from Theorem~6.2 in~\cite{bickel2012semiparametric},
it suffices to show that the conditional posterior of~$\vect\theta$
contracts uniformly (with respect to~$f \in  \mathcal{H}_n$) at parametric rate, i.e.~that
it holds
\begin{equation}\label{eq:marginalParametricFlat:toShow}
\sup_{f\in\mathcal H_n}\widehat{P}_{W;\tau_0,\tau_1}^n\left( \vect\theta\colon \|\matr X_n(\vect\theta-\vect\theta_0)\|\geq M_n|f,\vect Z_n \right) \to 0
\end{equation}
in $P_{W;0}^n$-probability for every sequence~$(M_n)$ with~$M_n\to\infty$.
First recall that under the flat prior on~$\vect\theta$, the conditional
posterior is equal to the conditional likelihood.
It then follows by analogous arguments as in Section 9.2.2 of \cite{christensen2011bayesian} that~$\widehat{P}_{W;\tau_0,\tau_1}^n(\cdot|f,\vect Z_n)$
has an~$r$-variate normal distribution with mean $\vecth\theta_n$ and covariance matrix $\matr \Sigma_n$ where
\begin{equation}\label{eq:hatthetan}
  \vecth\theta_n:=\matr \Sigma_n\matrt X_n^T\matr D_n^{-1}\vectt Z_n \in \mathbb R^r, \quad
  \matr \Sigma_n   :=(\matrt X_n^T\matr D_n^{-1}\matrt X_n)^{-1}\in\mathcal S_r^+.
\end{equation}
Thus by Markov's inequality, we have
\begin{align}\label{eq:hat_prob_upperbound}
&\widehat{P}_{W;\tau_0,\tau_1}^n(\vect\theta\in\mathbb{R}^r\colon\|\matr X_n(\vect\theta-\vect\theta_0)\|\geq M_n|f,\vect Z_n)\notag\\
&\leqslant \frac{1}{M_n^2}\E_{\widehat{P}_{W;\tau_0,\tau_1}}^n (\vect\theta-\vect\theta_0)^T\matrt X_n^T\matrt X_n (\vect\theta-\vect\theta_0)\notag\\
&= \frac{1}{M_n^2}\tr(\matrt X_n \matr \Sigma_n\matrt X_n^T)
+\frac{1}{M_n^2}(\vecth\theta_n-\vect\theta_0)^T\matrt X_n^T\matrt X_n (\vecth\theta_n-\vect\theta_0).
\end{align}
In view of $f\in\mathcal{H}_n\subset\mathcal{C}_{\tau_0,\tau_1}$, we have by definition of $\matr\Sigma_n$, the cyclic property of the trace and Lemma~\ref{lemma:cfLike}(c), that
\begin{align*}
\tr(\matrt X_n \matr \Sigma_n\matrt X_n^T)\leqslant \tau_1\tr(\matr D_n^{-1/2}\matrt X_n \matr \Sigma_n\matrt X_n^T \matr D_n^{-1/2})=r\tau_1,
\end{align*}
which implies $\sup_{f\in\mathcal H_n}M_n^{-2}\tr(\matrt X_n \matr \Sigma_n\matrt X_n^T)\to 0$. 
As for the second term of \eqref{eq:hat_prob_upperbound}, first note  that $\vecth\theta_n-\vect\theta_0 = \matr \Sigma_n\matrt X_n^T\matr D_n^{-1}\vectt e_n$, where $\vectt e_n=\matr F_n \vect e_n\sim N(\vect 0, \matr D_{n}(f_0))$ under $P_{W;0}^n$-probability. Therefore,
\begin{align}
&(\vecth\theta_n-\vect\theta_0)^T\matrt X_n^T\matrt X_n (\vecth\theta_n-\vect\theta_0)
=\vectt e_n^T \matr D_n^{-1} (\matrt X_n\matr \Sigma_n\matrt X_n^T)^2 \matr D_n^{-1} \vectt e_n .\label{eq_le7_ck}
\end{align}
Let $\preceq$ denote the Loewner partial order defined on the collection of symmetric matrices, i.e.\ for symmetric matrices $\matr A$, $\matr B$, $\matr A\preceq\matr B$ if and only if $\matr B - \matr A$ are positive semidefinite. Starting at $\tau_0I_{n}\preceq\matr D_n\preceq\tau_1I_{n}$, simple calculation, which involves (7.4) and Theorem 7.8 (3) of \cite{zhang11_matrix_book}, yields
\begin{align}\label{eq:loewner_bound_on_sigma_n}
\tau_0(\matrt X_n^T\matrt X_n)^{-1}\preceq\matr\Sigma_n\preceq\tau_1(\matrt X_n^T\matrt X_n)^{-1}
\end{align}
for any $f\in\mathcal{H}_n$. Moreover,  $\matrt X_n\matr \Sigma_n\matrt X_n^T$ and $\matrt Q_n :=\matrt X_n(\matrt X_n^T\matrt X_n)^{-1}\matrt X_n^T$ commute, because
\begin{align*}
	&\matrt X_n\matr \Sigma_n\matrt X_n^T\matrt Q_n=
\matrt X_n\matr \Sigma_n\matrt X_n^T
	= \matrt Q_n\matrt X_n\matr \Sigma_n\matrt X_n^T.
\end{align*}
Combined with \eqref{eq:loewner_bound_on_sigma_n}, (7.4) of \cite{zhang11_matrix_book} and Exercise 12.19 (c) of \cite{abadir_magnus05}, this yields $(\matrt X_n\matr \Sigma_n\matrt X_n^T)^2\preceq(\tau_1\matrt Q_n)^2=\tau_1^2\matrt Q_n$ and subsequently together with \eqref{eq_le7_ck}
\begin{align*}
&M_n^{-2}(\matrt X_n^T\matrt X_n)(\vecth\theta_n-\vect\theta_0)^T\matrt X_n^T\matrt X_n (\vecth\theta_n-\vect\theta_0)
\leqslant \tau_1^2 M_n^{-2}\vectt e_n^T \matr D_n^{-1} \matrt Q_n \matr D_n^{-1} \vectt e_n \notag\\
&=\tau_1^2 M_n^{-2}\|(\matr X_n^T\matr X_n)^{-\frac{1}{2}}\matrt X_n^T \matr D_n^{-1} \vectt e_n\|^2 
\leqslant \tau_1^2 M_n^{-2} \|(\matr X_n^T\matr X_n)^{-\frac{1}{2}}\|^2 \|\matrt X_n^T \matr D_n^{-1} \vectt e_n\|^2 \notag\\
&=\tau_1^2 M_n^{-2}\lambda_{\min}^{-1}(\matr X_n^T\matr X_n)\|\matrt X_n^T \matr D_n^{-1} \vectt e_n\|^2.
\end{align*}
Denoting  $\matr W_n(f):=n^{-1/2}M_n^{-1}\matrt X_n^T \matr D_n^{-1}(f)$, where unlike above we emphasize the dependence of $\matr D_n=\matr D_n(f)$ on $f$, by  Assumption~\ref{ass:designMatrix} we continue to upper bound
\begin{align}\label{eq:quad_form_upperbound}
&M_n^{-2}\lambda_{\min}^{-1}(\matr X_n^T\matr X_n)\|\matrt X_n^T \matr D_n^{-1} \vectt e_n\|^2 
= \frac{n}{\lambda_{\min}(\matr X_n^T\matr X_n)}\,\|\matr W_n(f) \vectt e_n\|^2 \notag\\
&=O(1)(\|\matr W_n(f_0) \vectt e_n\|^2 + \|\matr W_n(f) \vectt e_n - \matr W_n(f_0) \vectt e_n\|^2 ).
\end{align}

Regarding the first term of \eqref{eq:quad_form_upperbound}, we get by another application of Lemma~\ref{lemma:cfLike}(c)
\begin{align*}
&\E_{W;0}^n \|\matr W_n(f_0) \vectt e_n\|^2
= n^{-1}M_n^{-2}\E_{W;0}^n \vectt e_n^T \matr D_n^{-1}(f_0)\matrt X_n \matrt X_n^T \matr D_n^{-1}(f_0) \vectt e_n \\
&=n^{-1}M_n^{-2}\tr\left(\matr D_n^{-1}(f_0)\matrt X_n \matrt X_n^T\right)
=O(n^{-1}M_n^{-2})\tr(\matrt X_n^T\matrt X_n) \\
&=O(n^{-1}M_n^{-2})\lambda_{\max}(\matrt X_n^T\matrt X_n)=O(M_n^{-2}),
\end{align*}
where the last equality is due to Assumption~\ref{ass:designMatrix}.

It remains to prove $\sup_{f\in\mathcal H_n}\|\matr W_n(f) \vectt e_n - \matr W_n(f_0) \vectt e_n\|^2\to 0$ in $P_{W;0}^n$-probability. To this end, we will employ Lemma~\ref{lemma:uniform_convergence}.  Indeed, for any $f,g\in\mathcal{H}_n$, we have  that
\begin{align*}
&\|W_n(f) - W_n(g)\|_{F}
=n^{-1/2}M_n^{-1}\|\matrt X_n^T(\matr D_n^{-1}(f) - \matr D_n^{-1}(g))\|_{F} \\
&=n^{-1/2}M_n^{-1}\sqrt{\tr(\matrt X_n^T(\matr D_n^{-1}(f) - \matr D_n^{-1}(g))^2\matrt X_n)}
= O(n^{-1/2}M_n^{-1})\sqrt{\tr(\matrt X_n^T\matrt X_n)}\|f-g\|_{\infty} \\
&=O(n^{-1/2}M_n^{-1})\sqrt{\lambda_{\max}(\matr X_n^T\matr X_n)}\|f-g\|_{\infty}
=O(M_n^{-1})\|f-g\|_{\infty},
\end{align*}
where the last equality is  due to Assumption~\ref{ass:designMatrix}. This completes the proof by Lemma~\ref{lemma:uniform_convergence}.
\end{proof}
\end{Lemma}

We state the result of the lemma below w.r.t.\ the Whittle measure, in line with the proof of Theorem~\ref{eq:bvmLinearModel}. By contiguity, the assertion also holds w.r.t.\ $P_0^n$-probability.

\begin{Lemma}\label{lemma:marginalParametric}
Under the assumptions of Theorem~\ref{eq:bvmLinearModel}, it holds
\[
  P_{W;\tau_0,\tau_1}^n\left( \{\vect\theta\in\mathbb R^r\colon\|\matr X_n(\vect\theta-\vect\theta_0)\|\geq M_n\}\times\mathcal{H}_n|\vect Z_n \right) \to 0
\]
in~$P_{W;0}^n$-probability as~$n\to\infty$, where $\mathcal H_n$ is defined in \eqref{eq:ThetaNHNDefH} and $M_n$ as in Lemma~\ref{lemma:marginalParametricFlat}.
\end{Lemma}

\begin{proof}
For the sake of enhanced proof readability, we drop the subscripts~$\tau_0,\tau_1$ in the notation
and write~$P_W^n$ for the posterior and~$P(df)$ for the marginal prior.
Let $\Gamma_n = \{\vect\theta\in\mathbb R^r\colon\|\matr X_n(\vect\theta-\vect\theta_0)\|\geq M_n\}$, $\rho_n:=\int_{\Gamma_n\times\mathcal{H}_n} p_W^n(\vect Z_n|\vect\theta,f)P(d\vect\theta)P(df)$ and $\hat\rho_n=\int_{\Gamma_n\times\mathcal{H}_n}p_W^n(\vect Z_n|\vect\theta,f)d\vect\theta P(df)$.
By Assumption~\ref{ass:thetaPrior}
it holds 
\begin{align*}
  &P_W^n((\vect\theta ,f)\in\Gamma_n\times\mathcal H_n|\vect Z_n) 
  = \frac{1}{\rho_n}\int_{\Gamma_n\times\mathcal{H}_n}p_W^n(\vect Z_n|\vect\theta,f)P(d\vect\theta)P(df)\\
  &
  \leq \frac{C\hat\rho_n}{\rho_n}\,\widehat{P}_{W;\tau_0,\tau_1}^n((\vect\theta ,f)\in\Gamma_n\times\mathcal H_n|\vect Z_n)
\end{align*}
By Lemma~\ref{lemma:marginalParametricFlat} it holds ~$\widehat{P}_{W;\tau_0,\tau_1}^n((\vect\theta ,f)\in\Gamma_n\times\mathcal H_n|\vect Z_n)\to 0$
in~$P_{W;0}^n$-probability as~$n\to\infty$, such that it is sufficient to show that the evidence
ratio~$\hat\rho_n/\rho_n$ is bounded in~$P_{W;0}^n$-probability.

By some straightforward calculations (see e.g.\ the steps of~(5) in Section~9.2.1 in~\cite{christensen2011bayesian})  we obtain the following representation for the quadratic expression in 
the exponent of Whittle's likelihood~$p_W^n$ from~\eqref{eq:whittleLinearModel}:
\begin{align*}
  &\left( \vectt Z_n-\matrt X_n\vect\theta \right)^T\matr D_n^{-1}\left( \vectt Z_n-\matrt X_n\vect\theta \right)\\
  &= \left( \vectt Z_n-\matrt X_n\vecth\theta_n \right)^T\matr D_n^{-1}\left( \vectt Z_n-\matrt X_n\vecth\theta_n \right) + \left( \vect\theta-\vecth\theta_n \right)^T \matr\Sigma_n^{-1} \left( \vect\theta-\vecth\theta_n \right)
\end{align*}
with~$\vecth\theta_n$ and~$\matr\Sigma_n$ from~\eqref{eq:hatthetan}.
Observing
\[
  \int\exp\left( -\frac{1}{2}\left( \vect\theta-\vecth\theta_n \right)^T \matr\Sigma_n^{-1} \left( \vect\theta-\vecth\theta_n \right) \right)d\vect\theta = \sqrt{(2\pi)^r|\matr\Sigma_n|},
\]
this yields, by Tonelli's theorem,
\begin{align*}  \hat\rho_n 
  &=\int\int p_W^n(\vect Z_n|\vect\theta,f)d\vect\theta P(df)\notag\\
  &=  \left( \frac{|\matr\Sigma_n|}{|\matr D_n|} \right)^{1/2}\int\exp\left( -\frac{1}{2}\left( \vectt Z_n-\matrt X_n\vecth\theta_n \right)^T\matr D_n^{-1}\left( \vectt Z_n-\matrt X_n\vecth\theta_n \right) \right)P(df)\notag
\end{align*}
Similarly, 
for~$\rho_n$, with $P( d\vect\theta)=g(\vect\theta) d\vect\theta$ by Assumption~\ref{ass:thetaPrior},
\begin{align*}
	\rho_n&=\hat\rho_n\, \frac{1}{\sqrt{2\pi |\matr\Sigma_n|}}\, \int\exp\left( -\frac{1}{2}\left( \vect\theta-\vecth\theta_n \right)^T \matr\Sigma_n^{-1} \left( \vect\theta-\vecth\theta_n \right) \right)g(\vect\theta) d\vect\theta\\
	&=\hat\rho_n\,\E_{\vect\xi}\left( g(\vecth\theta_n+\matr\Sigma_n^{1/2}\vect\xi )\right),
\end{align*}
where $\vect\xi\sim N_r(\vect 0,\matr I_r)$ and $\E_{\vect\xi}$ denotes the expectation with respect to $P_{\vect\xi}$.
By Assumption~\ref{ass:thetaPrior}, there exists~$\varepsilon>0$ such that
it holds $g(\vect\theta)\geq c\, \indi_{\{ \|\vect\theta-\vect\theta_0\|<\varepsilon\}}$ for some $c>0$. Consequently,
\begin{align*}
	&\E_{\vect\xi}\left( g(\vecth\theta_n+\matr\Sigma_n^{1/2}\vect\xi )\right)
	\geq c\,  P_{\vect\xi}\left(\| \matr\Sigma_n^{1/2}\vect\xi+ \vecth\theta_n-\vect\theta_0\|<\varepsilon \right).\end{align*}
Furthermore, with $ \matr\Sigma_{n,0}= (\matrt X_n^T\matr D_{n,0}^{-1}\matrt X_n)^{-1}$, by applications of Lemma~\ref{lemma:cfLike}
\begin{align*}
	&\| \matr\Sigma_n^{1/2}\|^2=\lambda_{\max}\left( (\matrt X_n^T\matr D_{n}^{-1}\matrt X_n)^{-1} \right)=\left(\lambda_{\min}\left( \matrt X_n^T\matr D_{n}^{-1}\matrt X_n \right)  \right)^{-1}\leq \frac{\tau_1}{\lambda_{\min}(\matr X_n^T\matr X_n)},
\end{align*}
such that
\begin{align*}
	&\| \matr\Sigma_n^{1/2}\vect\xi+ \vecth\theta_n-\vect\theta_0\|\leq \| \matr\Sigma_n^{1/2}\|\,\| \vect\xi\|+ \| \matr\Sigma_{n,0}^{1/2}\| \,\| \matr\Sigma_{n,0}^{-1/2}(\vecth\theta_n-\vect\theta_0)\|\\
	&\le \sqrt{\frac{\tau_1}{\lambda_{\min}(\matr X_n^T\matr X_n)}}\,\| \vect\xi\|
+	\sqrt{\frac{\tau_1}{\lambda_{\min}(\matr X_n^T\matr X_n)}}\,\| \matr\Sigma_{n,0}^{-1/2}(\vecth\theta_n-\vect\theta_0)\|.
\end{align*}
Under $P_{W;0}^n$ it holds
$\matr \Sigma_{n,0}^{-1/2}(\vecth\theta_n-\vect\theta_0)\sim N(\vect 0,\matr I_r)$ and $\lambda_{\min}(\matr X_n^T\matr X_n)\to\infty$ by Assumptions~\ref{ass:designMatrix}, such that
\begin{align*}
	&	\sqrt{\frac{\tau_1}{\lambda_{\min}(\matr X_n^T\matr X_n)}}\,\| \matr\Sigma_{n,0}^{-1/2}(\vecth\theta_n-\vect\theta_0)\|=o_{P_{W;0}^n}(1).
\end{align*}
Consequently,
\begin{align*}
	&P_{\vect\xi}\left(\| \matr\Sigma_n^{1/2}\vect\xi+ \vecth\theta_n-\vect\theta_0\|<\varepsilon \right)\ge  P_{\vect\xi}\left( \|\xi\|<	\sqrt{\frac{\lambda_{\min}(\matr X_n^T\matr X_n)}{\tau_1}} \left( \varepsilon+ o_{P_{W;0}^n}(1)\right) \right)\\
	&=1+o_{P_{W;0}^n}(1),
\end{align*}
where the last statement follows e.g.\ from an application of the Markov inequality and $\lambda_{\min}(\matr X_n^T\matr X_n)\to\infty$.
Putting everything together, we get for the evidence ratio
\begin{align*}
	&\frac{\hat\rho_n}{\rho_n}\le \frac{1}{c\, (1+o_{P_{W;0}^n}(1))}=O_{P_{W;0}^n}(1),
\end{align*}
completing the proof.
\end{proof}
\subsection{Joint conditional nuisance posterior concentration}

We can now put the previous results together to prove that the joint conditional nuisance posterior concentration as in Theorem~\ref{th:bvmMultiGeneral} (c) holds in this setting, where the statement also holds in $P_0^n$-probability due to contiguity.

\begin{Theorem}\label{th:jointContr}
Under the assumptions of Theorem~\ref{eq:bvmLinearModel}, it holds
\[
  P_{W;\tau_0,\tau_1}^n((\vect\theta ,f)\in\Theta_n\times\mathcal H_n|\vect Z_n)\to 1
\]
in~$P_{W;0}^n$-probability as~$n\to \infty$ for~$\Theta_n=
	\left\{ \vect\theta\in\Theta\colon(\vect\theta-\vect\theta_0)^T\matr S_n(\vect\theta-\vect\theta_0)\leq \sqrt{\lambda_{\min}(\matr S_n)}	\right\}$, $\matr S_n=\matrt{X}_n^T\matr D_{n,0}^{-1}\matrt{X}_n$ and $\mathcal H_n$ 
as in~\eqref{eq:ThetaNHNDefH} .
\begin{proof}
From Theorem~\ref{lemma:jointHellinger}, we know that~the joint posterior with respect to the root average squared
Hellinger distance~$d_{n,H}$ (as defined in~\eqref{eq:hellingerDefSemi}) contracts polynomially.
Thus, it follows from Corollary~\ref{lemma:hellingerFromBelowBvm} in the Appendix that the
Hellinger ball around~$(\vect\theta_0,f_0)$ of radius~$n^{-s}$
is a subset of~$\mathbb R^r\times\mathcal H_n$ for~$n$ large enough, yielding
\[
  P_{W;\tau_0,\tau_1}^n((\vect\theta ,f)\in\mathbb R^r\times\mathcal H_n|\vect Z_n)\to 1
\]
in~$P_{W;0}^n$-probability.
First, it holds
\begin{align*}
	&(\vect\theta-\vect\theta_0)^T\matr S_n(\vect\theta-\vect\theta_0)\leq \frac{1}{\tau_0}\, (\vect\theta-\vect\theta_0)^T\matr X_n^T\matr X_n(\vect\theta-\vect\theta_0),
\end{align*}
such that 
\begin{align*}
	&P_{W;\tau_0,\tau_1}^n(\Theta_n\times\mathcal{H}_n|\vect Z_n)
	\geq
	P_{W;\tau_0,\tau_1}^n(\{\|\matr X_n  (\vect\theta-\vect\theta_0)\|^2< \tau_0 \lambda_{\min}^{1/2}(\matr S_n)\}\times\mathcal{H}_n|\vect Z_n) \\
	&=P_{W;\tau_0,\tau_1}^n(\mathbb R^r\times\mathcal H_n|\vect Z_n) - P_{W;\tau_0,\tau_1}^n(\{\|\matr X_n  (\vect\theta-\vect\theta_0)\|\geq \sqrt{\tau_0}\lambda_{\min}^{1/4}(\matr S_n)\}\times\mathcal{H}_n|\vect Z_n)\\
	&\to 1
\end{align*}
in~$P_{W;0}^n$-probability by \eqref{eq_lambda_mintoinfty} and Lemma~\ref{lemma:marginalParametric}. 
\end{proof}
\end{Theorem}

\section{Discussion}
In this paper, 
a semiparametric Bernstein-von-Mises theorem for the parameter vector of a linear model with nonparametrically modelled time series errors has been derived where the pseudo-posterior distribution is based on the Whittle likelihood approximation.

The theorem shows that the pseudo-posterior is asymptotically Gaussian with mean  equal to the pseudo-maximum likelihood estimator calculated under the Whittle likelihood and covariance matrix equal to the asymptotic covariance matrix of the pseudo-maximum likelihood estimator also calculated under the Whittle likelihood. 

The theorem is crucial as Bayesian uncertainty quantification e.g.\ by means of credibility sets or posterior standard deviations will only report the true frequentist uncertainty  if  this  covariance matrix asymptotically coincides in a suitable sense with the corresponding quantity under the exact Gaussian likelihood. 

We discuss conditions under which this asymptotic equivalence holds and illustrate this in  examples. However, we also provide a counter-example where these conditions are not met  and where posterior inference based on the asymptotic Gaussian approximation under the Whittle likelihood  results in a systematic over- or underestimation of the statistical uncertainty.

An important avenue for future research will be to obtain a better understanding of what assumptions are required to have not only first (in the sense of posterior consistency) but also second-order correctness (in the sense that the asymptotic covariance associated with a Bayesian point estimator coincides with the one present in the posterior).

\section*{Acknowledgements}
The work was supported by DFG grants KI 1443/3-1 and KI 1443/3-2.
R.M.\ also acknowledges support by the Marsden grant MFP-UOA2131 from New Zealand Government funding, administered by the Royal Society Te Ap\=arangi. We thank the New Zealand eScience Infrastructure
(NeSI) \url{https://www.nesi.org.nz} for the use of their high performance computing facilities and the Centre for eResearch at the University of Auckland for their technical
support.

\setlength{\bibsep}{0.0pt}
\bibliographystyle{abbrvnat}

\bibliography{draft}
\newpage
\appendix
\section*{Appendix}
\section{Technical lemmas}\label{sec:technical}

\begin{Lemma}\label{lemma:cfLike}
Let~$\matr A\in\mathbb R^{n\times r}$ and~$\matr B,C\in\mathcal S_n^+$.
Then it holds
\begin{align*}
 &(a)\quad \lambda_{\min}(\matr A^T\matr B\matr A) \geq \lambda_{\min}(\matr A^T\matr A)\lambda_{\min}(\matr B), \\
&(b)\quad  \lambda_{\max}(\matr A^T\matr B\matr A) \leq \lambda_{\max}(\matr A^T\matr A)\lambda_{\max}(\matr B),\\
&(c)\quad \operatorname{tr}(BC)\le \lambda_{\max}(B)\,\operatorname{tr}(C).
\end{align*}
\begin{proof}
	The result for (a) and (b) readily follows from the Min-Max principle of Courant-Fisher, (c) is Lemma B.4(d) in \cite{meier}.
\end{proof}
\end{Lemma}

The following result
bounds the KL terms between multivariate normals from above in terms of distances
of the respective matrices and mean vectors (a more detailed proof can be found in ~\cite{meier}, Lemma 9.2). \begin{Lemma}\label{lemma:KLhelpRealValued}
Let~$\vect \mu_0,\vect \mu_1 \in \mathbb R^d$ and~$\matr \Sigma_0,\matr \Sigma_1 \in \mathcal S_d^+(\mathbb R)$,
where~$\mathcal S_d^+(\mathbb R)$ denotes the cone of symmetric positive definite
matrices in~$\mathbb R^{d\times d}$,
with~$\lambda_{\min}(\matr \Sigma_i)\geq\tau_0$ and~$\lambda_{\max}(\matr \Sigma_i)\leq\tau_1$
for~$i=0,1$ and some positive constants~$\tau_0,\tau_1$.
Assume that~$\lambda_{\min}(\matr\Sigma_1^{-1/2}\matr\Sigma_0 \matr\Sigma_1^{-1/2})\geq\frac{1}{2}$.
Let~$p_i$ denote the density of the~$N_d(\vect \mu_i,\matr \Sigma_i)$ distribution for~$i=0,1$.
Denote by
\[
  K:=\int_{\mathbb R^d} \log\frac{p_0(\vect z)}{p_1(\vect z)}p_0(\vect z)d\vect z, \quad
  V:=\int_{\mathbb R^d} \left( \log\frac{p_0(\vect z)}{p_1(\vect z)} - K \right)^2p_0(\vect z)d\vect z
\]
the Kullback-Leibler divergence and associated variance term from~$p_0$ to~$p_1$.
Then it holds
\begin{align*}
  K \lleq \| \matr \Sigma_0-\matr \Sigma_1 \|_F^2 + \|\vect \mu_0-\vect \mu_1\|^2
\end{align*}
and
\begin{align*}  V \lleq \| \matr \Sigma_0-\matr \Sigma_1 \|_F^2   + \|\vect \mu_0-\vect \mu_1\|^2
\end{align*}
with proportionality constants only depending on~$\tau_0$ and $\tau_1$.
\begin{proof}
Some calculations yield for~$\vect z\in\mathbb R^d$
\begin{align}  2\log\frac{p_0(z)}{p_1(z)} 
  &= \left[\log|\matr \Sigma_1| - \log|\matr \Sigma_0| + (\vect z-\vect \mu_0)^T \left(\matr \Sigma_1^{-1}-\matr \Sigma_0^{-1}\right) (\vect z-\vect \mu_0)\right] \notag\\
  &\quad     
    + 2(\vect \mu_0-\vect \mu_1)^T \matr \Sigma_1^{-1} (\vect z-\vect \mu_0) 
    + (\vect \mu_0-\vect \mu_1)^T \matr \Sigma_1^{-1} (\vect \mu_0-\vect \mu_1)\notag\\
    &=:l_1(\vect z)+2\, l_2(\vect z) + (\vect \mu_0-\vect \mu_1)^T \matr \Sigma_1^{-1} (\vect \mu_0-\vect \mu_1).
   \label{eq:KLrealMultiHelp}
\end{align}
Since~$l_1(\vect z)/2$  is the log-likelihood-ratio between an~$N_d(\vect \mu_0,\matr\Sigma_0)$ and an~$N_d(\vect \mu_0,\matr\Sigma_1)$ distribution, we conclude (see e.g.\ Lemma~7.5 in \cite{meier}) for the corresponding Kullback-Leibler divergence
\[
  \int l_1(\vect z)p_0(\vect z)dz 
  \leq \frac{2}{\lambda_{\min}(\matr\Sigma_1)^{2}}\|\matr \Sigma_0-\matr \Sigma_1\|_F^2
  \lleq \|\matr \Sigma_0-\matr \Sigma_1\|_F^2.
\]
By $\int l_2(\vect z)p_0(\vect z)d\vect z=0$ and $(\vect \mu_0-\vect \mu_1)^T \matr \Sigma_1^{-1} (\vect \mu_0-\vect \mu_1)\lleq\| \vect \mu_0-\vect \mu_1 \|^2$, the first assertion follows.

\par
For the second assertion, let~$\vect Z \distr N_d(\vect 0,\matr\Sigma_0)$.
Then~$V=\Var[\log\frac{p_0(\vect Z+\mu_0)}{p_1(\vect Z+\mu_0)}]$ and from~\eqref{eq:KLrealMultiHelp} this yields \[
  4V = \Var \left[ \vect Z^T\left( \matr \Sigma_1^{-1}-\matr \Sigma_0^{-1}\right)\vect Z + 2(\vect \mu_0-\vect \mu_1)^T\matr \Sigma_1^{-1}\vect Z  \right]
    =: \Var \left[ \eta_1+2\eta_2 \right].
\]
$\Var[\eta_1]$ is the KL variance term  (up to a multiplicative constant)
between an~$N_d(\vect 0,\matr\Sigma_0)$ and an~$N_d(\vect 0,\matr\Sigma_1)$
distribution, so that (see e.g.\ Lemma~7.5 in \cite{meier})
\[
  \Var[\eta_1]
  \lleq \frac{1}{\lambda_{\min}(\matr\Sigma_1)^{2}}\|\matr \Sigma_0-\matr \Sigma_1\|_F^2
  \lleq \|\matr \Sigma_0-\matr \Sigma_1\|_F^2.
\]
Furthermore,
\begin{align*} 
 & \Var[\eta_2]
  = \Var\left[ (\vect \mu_0-\vect \mu_1)^T\matr \Sigma_1^{-1}\vect Z \right]
  = \E \left[ (\vect \mu_0-\vect \mu_1)^T\matr \Sigma_1^{-1}\vect Z\vect Z^T\matr \Sigma_1^{-1}(\vect \mu_0-\vect \mu_1) \right]  \\
  &=  (\vect \mu_0-\vect \mu_1)^T\matr \Sigma_1^{-1}\matr \Sigma_0\matr \Sigma_1^{-1}(\vect \mu_0-\vect \mu_1)\leq \lambda_{\max}\left( \matr \Sigma_1^{-1}\matr \Sigma_0\matr \Sigma_1^{-1} \right) \|\vect \mu_0-\vect \mu_1\|^2\lleq \|\vect \mu_0-\vect \mu_1\|^2.
\end{align*}
The assertion follows by $|\Cov(\eta_1,\eta_2)|\le \max(\Var(\eta_1),\Var(\eta_2))$ by an application of the Cauchy-Schwarz inequality.\end{proof}
\end{Lemma}
The following result bounds the Hellinger distance between two  normals.
A corresponding upper bound for multivariate normals can be found in Lemma~9.4 of \cite{meier} while in this situation the lower bound can be obtained similarly to Lemma~7.16 in \cite{meierPaper}.

\begin{Lemma}\label{lemma:hellingerBoundRealValuedSemiparametric}
Let~$\mu_1,\mu_2 \in \mathbb R$ and~$0<\tau_0\le \sigma_i^2\le \tau_1<\infty$, $i=1,2$, for some positive constants~$\tau_0,\tau_1$.
Let~$p_i$ denote the density of the~$N(\mu_i,\sigma_i^2)$ distribution for~$i=1,2$.
Then it holds
\begin{align*}
	|\sigma_1^2-\sigma_2^2|^2+ 1-\exp\left(-c|\mu_1-\mu_2|^2\right)
\lleq
d_H^2(p_1,p_2) \lleq |\sigma_1^2-\sigma_2^2| + |\mu_1-\mu_2|
\end{align*}
with $c>0$ and proportionality constants depending only on~$\tau_0,\tau_1$.
\begin{proof}
We start by noting the representation (see~\cite{pardo05}, p.51 and 46)
\begin{align}
  d_H^2(p_1,p_2) 
  &= 1 - \sqrt{\frac{2\sigma_1\sigma_2}{\sigma_1^2+\sigma_2^2}}\,\exp\left( -\frac{1}{4}\,\frac{(\mu_1-\mu_2)^2}{\sigma_1^2+\sigma_2^2} \right)\notag\\
  &=\left(1-\sqrt{\frac{2\sigma_1\sigma_2}{\sigma_1^2+\sigma_2^2}}\right)+\sqrt{\frac{2\sigma_1\sigma_2}{\sigma_1^2+\sigma_2^2}}\left( 1-\exp\left( -\frac{1}{4}\,\frac{(\mu_1-\mu_2)^2}{\sigma_1^2+\sigma_2^2} \right)\right).\label{eq_decomp_Hellinger}
\end{align}
For the first summand we obtain the upper bound
\begin{align*}
	&0\le 1-\sqrt{\frac{2\sigma_1\sigma_2}{\sigma_1^2+\sigma_2^2}}\le \sqrt{\frac{(\sigma_1-\sigma_2)^2}{\sigma_1^2+\sigma_2^2}}\le \frac{1}{\sqrt{\sigma_1^2+\sigma_2^2}}\,\frac{1}{2 \sqrt{\tau_0}}\,|\sigma_1^2-\sigma_2^2|\lleq |\sigma_1^2-\sigma_2^2|,
\end{align*}
where in the second to last step we used the mean value theorem to bound $|\sqrt{x_1}-\sqrt{x_2}|$  (with $x_i=\sigma_i^2$).
For the second term we get 
\begin{align*}
	&\sqrt{\frac{2\sigma_1\sigma_2}{\sigma_1^2+\sigma_2^2}}\left( 1-\exp\left( -\frac{1}{4}\,\frac{(\mu_1-\mu_2)^2}{\sigma_1^2+\sigma_2^2} \right)\right)\lleq 1-\exp\left( -\frac{1}{4}\,\frac{(\mu_1-\mu_2)^2}{\sigma_1^2+\sigma_2^2}\right) \\
	&\leq 1-\exp\left( -\frac{1}{8\,\tau_0}\,(\mu_1-\mu_2)^2\right) \lleq |\mu_1-\mu_2|,
\end{align*}
because
the function~$h_c$ with $h_c(x) := 1-\exp(-cx^2)$
has points of inflection at~$x^*=\pm (2c)^{-1/2}$, where the maximal derivative
is attained, with value~$h_c'(x^*)=\sqrt{2c}\exp(-1/2)$. This completes the proof of the upper bound.
\par

Because $1-\sqrt{x}\geq (1-x)/2$ we get the lower bound for the first summand in \eqref{eq_decomp_Hellinger}
\begin{align*}
	&1-\sqrt{\frac{2\sigma_1\sigma_2}{\sigma_1^2+\sigma_2^2}}\geq \frac 1 2 \,\frac{(\sigma_1-\sigma_2)^2}{\sigma_1^2+\sigma_2^2}\ggeq |\sigma_1^2-\sigma_2^2|^2,
\end{align*}
where, in the last step,  we used again the mean value theorem to bound $|\sqrt{x_1}-\sqrt{x_2}|$  (with $x_i=\sigma_i^2$) this time from below. The lower bound for the second summand in \eqref{eq_decomp_Hellinger} follows straightforwardly from 
$\tau_0\le \sigma_i^2\le \tau_1$.
\end{proof}
\end{Lemma}

The previous lower bound carries over to a lower bound
for the average squared Hellinger distance~$d_{n,H}^2((\vect\theta, f),(\vect\theta_0, f_0))$ (see~\eqref{eq:hellingerDefSemi}), 
as summarized in the following corollary.
\begin{Corollary}\label{lemma:hellingerFromBelowBvm}
Let~$ f \in \mathcal C_{\tau_0,\tau_1}$ with~$\mathcal C_{\tau_0,\tau_1}$
from~\eqref{eq:truncationSetForBvm}. Then it holds
\[
  d_{n,H}^2((\vect\theta, f),(\vect\theta_0, f_0))
  \ggeq \| f- f_0\|^3_{\infty} + \sum_{j=1}^n\frac{1-\exp\left(-c|(\matrt X_n(\vect\theta-\vect\theta_0))_j|^2\right)}{n} +O\left( \frac 1 n \right).
\]
\begin{proof}
The proof is analogous to the proof of Lemma~7.17 in \cite{meierPaper},
using the lower bound from Lemma~\ref{lemma:hellingerBoundRealValuedSemiparametric}.
\end{proof}
\end{Corollary}

\begin{Lemma}\label{lemma:uniform_convergence}
Denote $(\Omega, \mathcal{F}, P)$ a probability space in which all the random elements considered below are defined. $E$ is the corresponding expectation. 

Let $\matr W_n(\cdot)=(w_{i,j;n}(\cdot))_{i=1,\cdots,r;j=1\cdots,n} $ be an $r\times n$ matrix, where $r\in\mathbb{N}$ is a constant and $w_{i,j;n}:\mathcal{H}\to\mathbb{R}$ with $\mathcal{H}=\mathcal{C}_{\tau_0,\tau_1}$ defined in \eqref{eq:truncationSetForBvm}. Suppose for each $n\in\mathbb{N}$, $\vect\xi_{n}=(\xi_{1;n},\cdots,\xi_{n;n})^{T}\sim N(\vect 0,\matr\Sigma_n)$, where $\matr\Sigma_n=\diag(\sigma_{1;n}^2,\cdots,\sigma_{n;n}^2)$ with $\sigma_{i;n}^2\leq M$ for any $i=1,\cdots,n$, $n\in\mathbb{N}$ and $M>0$ is a constant. If
\begin{itemize}
\item[(i) ] $\|\matr W_n(f) - \matr W_n(g)\|_{F}\leq A_{n}\|f-g\|_{\infty}$ for any $f,g\in\mathcal{H}$;
\item[(ii) ] $A_n=o\left( \sqrt{\log n} \right)$ when $n\to\infty$.
\end{itemize}
Then we have
\begin{align*}
\sup_{f\in\mathcal{H}_{n}}\|\matr W_n(f)\vect\xi_{n}-\matr W_n(f_0)\vect\xi_{n}\|\to0
\end{align*}
in $P$-probability, where $\mathcal{H}_n$ is defined in \eqref{eq:ThetaNHNDefH}.
\begin{proof}
Denote $Q_{i;n}(\cdot)=\sum_{j=1}^{n}w_{i,j;n}(\cdot)\xi_{j;n}$, $i=1,\cdots,r$. Due to the following equality
\begin{align*}
\|\matr W_n(f)\vect\xi_{n}-\matr W_n(f_0)\vect\xi_{n}\|^2=\sum_{i=1}^r (Q_{i;n}(f) - Q_{i;n}(f_0))^2,
\end{align*}
we only need to prove $E\sup_{f\in\mathcal{H}_{n}}|Q_{i;n}(f) - Q_{i;n}(f_0)|\to0$ for each $i=1,\cdots,r$. It is easily seen that $Q_{i;n}(f)$ is a centered Gaussian process indexed by $f\in\mathcal{H}$. We can always choose $Q_{i;n}$ to be a separable version in view of the fact that $(\mathcal{H},\|\cdot\|_{\infty})$ is compact (cf. Section 2.3.3 of \cite{van_der_vaart_and_wellner1996})). Denote $\rho_{i;n}(f,g)=\sqrt{E(Q_{i;n}(f) - Q_{i;n}(g))^2}$ the canonical metric induced by $Q_{i;n}$. Simple calculation yields
\begin{align*}
\rho_{i;n}(f,g)=\sqrt{\sum_{j=1}^n (w_{i,j;n}(f) - w_{i,j;n}(g))^2\sigma_{i;n}^2}\leq \sqrt{M}A_n\|f-g\|_{\infty}
\end{align*}
for any $f,g\in\mathcal{H}$, where the inequality is due to the boundedness of $\sigma_{i,n}^2$ and condition (i) stated above.

By using Dudley's bound (cf. Corollary 2.2.8 of \cite{van_der_vaart_and_wellner1996}) and the estimate of entropy number $N(\cdot,\cdot,\cdot)$ (see e.g. Theorem 2.7.1 of \cite{van_der_vaart_and_wellner1996}), we get
{\allowdisplaybreaks\begin{align*}
&E\sup_{f\in\mathcal{H}_{n}}|Q_{i;n}(f)-Q_{i;n}(f_0)| 
\leq E\sup_{\rho_{i;n}(f,g)\leq\sqrt{M}A_n(\log n)^{-1}}|Q_{i;n}(f)-Q_{i;n}(g)|\\
&\lleq \int_0^{\sqrt{M}A_n(\log n)^{-1}}\log^{\frac{1}{2}} N(\epsilon/2,\mathcal{H},\rho_{i;n})d\epsilon  \\
&
\leq \int_0^{\sqrt{M}A_n(\log n)^{-1}}\log^{\frac{1}{2}} N((2\sqrt{M}A_n)^{-1}\epsilon,\mathcal{H},\|\cdot\|_{\infty})d\epsilon  \\
&\lleq A_n\int_0^{(2\log n)^{-1}}\log^{\frac{1}{2}} N(\epsilon,\mathcal{H},\|\cdot\|_{\infty})d\epsilon \\
&
\lleq A_n\int_0^{(2\log n)^{-1}} \epsilon^{-\frac{1}{2}}d\epsilon
\lleq A_n(\log n)^{-\frac{1}{2}}=o(1),
\end{align*}}
where the last equality is due to condition (ii) stated above.
\end{proof}
\end{Lemma}

\end{document}

%% file: cmd.tex

\newcommand{\conj}[1]{\overline{#1}}
\newcommand{\adj}[1]{{#1}^*} 
\newcommand{\param}[1]{#1_{\operatorname{param}}}
\newcommand{\paramm}[2]{#1_{\operatorname{param};#2}}
\newcommand{\outerProduct}[1]{#1\adj{#1}}
\newcommand{\distr}{\sim}
\newcommand{\asEq}{\overset{\text{a.s.}}{=}}
\newcommand{\dEq}{\overset{d}{=}}
\newcommand{\iid}{\overset{\text{iid}}{\distr}}
\newcommand{\ind}{\overset{\text{ind.}}{\distr}}
\newcommand{\vech}{\operatorname{vec}}
\newcommand{\vecop}[1]{{{#1}}_{\operatorname{vec}}}
\newcommand{\veccop}[1]{{{#1}}_{\underline{\operatorname{vec}}}}
\newcommand{\diag}{\operatorname{diag}}
\newcommand{\Id}{\operatorname{Id}}
\newcommand{\Dir}{\operatorname{Dir}}
\newcommand{\tr}{\operatorname{tr}}
\newcommand{\block}{\mathcal B}
\newcommand{\etr}{\operatorname{etr}}
\newcommand{\lvy}{L\'{e}vy }
\newcommand{\frch}{Fr\'{e}chet }
\newcommand{\Exp}{\operatorname{Exp}}
\newcommand{\indi}{\mathds{1}}
\newcommand{\lleq}{\lesssim}
\newcommand{\ggeq}{\gtrsim}
\newcommand{\vardbtilde}[1]{\tilde{\raisebox{0pt}[0.85\height]{$\tilde{#1}$}}}
\newcommand{\matr}[1]{\bm{#1}} 
\newcommand{\matrt}[1]{\tilde{\bm{#1}}} 
\newcommand{\vect}[1]{\underline{#1}} 
\newcommand{\vectt}[1]{\underline{\tilde{#1}}} 
\newcommand{\vecth}[1]{\underline{\hat{#1}}} 
\newcommand{\ups}[1]{^{(#1)}}
\newcommand{\upss}[1]{^{[#1]}}
\newcommand{\cblue}{\color{blue}}
\newcommand{\cred}{\color{red}}

\newcommand{\modd}{\operatorname{mod}}
\newcommand{\MF}{\mathcal{MF}}

\newcommand{\E}{\mathrm E}
\newcommand{\Var}{\mathrm{Var}}
\newcommand{\Cov}{\mathrm{Cov}}
\newcommand{\DFT}{\mathcal F_n}
\newcommand{\iDFT}{\DFT^{-1}}